\documentclass[reqno,final]{amsart}
\usepackage{natbib}  
\usepackage{fancyhdr} 
\usepackage{color} 
\usepackage{hyperref} 
\usepackage{graphicx} 

\usepackage{pstricks}
\usepackage{amssymb}

\definecolor{aleacolor}{rgb}{0.16,0.59,0.78}

\hypersetup{
breaklinks,
colorlinks=true,
linkcolor=aleacolor,
urlcolor=aleacolor,
citecolor=aleacolor}

\renewcommand{\headheight}{11pt}
\renewcommand{\cite}{\citet}

\theoremstyle{plain}
\newtheorem{theorem}{Theorem}[section]                                          
\newtheorem{proposition}[theorem]{Proposition}

\theoremstyle{definition}
\newtheorem{definition}[theorem]{Definition}
\theoremstyle{remark}
\newtheorem{remark}[theorem]{Remark}

\makeatletter \@addtoreset{equation}{section} \makeatother
\renewcommand\theequation{\thesection.\arabic{equation}}
\renewcommand\thefigure{\thesection.\arabic{figure}}
\renewcommand\thetable{\thesection.\arabic{table}}

\newcommand{\R}{\mathbb{R}}

\begin{document}

\title[]{Large Deviation Principle for the Greedy\\ Exploration Algorithm over Erd\"os-R\'enyi Graphs}

\author{P. Bermolen}
\author{V. Goicoechea}
\author{M.Jonckheere}
\author{E.Mordecki}

\address{Instituto de Matem\'atica y Estad\'istica Prof. Rafael Laguardia, \newline
Facultad de Inegenier\'ia, \newline
Universidad de la Rep\'ublica,\newline
Uruguay.}
\email{paola@fing.edu.uy}

\address{Instituto de Matem\'atica y Estad\'istica Prof. Rafael Laguardia, \newline
Facultad de Inegenier\'ia, \newline
Universidad de la Rep\'ublica,\newline
Uruguay.}
\email{vgoicoechea@fing.edu.uy}

\address{Instituto de C\'alculo. Conicet, \newline
Facultad de Ciencias Exactas y Naturales, \newline
Universidad de Buenos Aires,\newline
Argentina.}
\email{mjonckhe@dm.uba.ar}

\address{Centro de Matem\'atica, \newline
Facultad de Ciencias, \newline
Universidad de la Rep\'ublica,\newline
Uruguay.}
\email{emordecki@cmat.edu.uy}


\thanks{Research supported by Stic Amsud Gene.}

\keywords{Large Deviation Principle, Greedy Exploration Algorithms, Erd\"os-R\'enyi Graphs, Hamilton-Jacobi equations, Comparison Principle}

\begin{abstract}
We prove a large deviation principle for a greedy exploration process on an Erd\"os-R\'enyi (ER) graph when the number of nodes goes to infinity. To prove our main result, we use the general strategy to study large deviations of processes proposed by \cite{F&K}, based on the convergence of non-linear semigroups. The rate function can be expressed in a closed-form formula, and associated optimization problems can be solved explicitly, providing the large deviation trajectory. Also, we derive an LDP for the size of the maximum independent set discovered by such an algorithm and analyze the probability that it exceeds known bounds for the maximal independent set.  We also analyze the link between these results and the landscape complexity of the independent set and the exploration dynamic.
\end{abstract}

\maketitle

\section{Introduction} \label{Introduction}
Consider a finite, possibly random, graph $G$ for which $V$ is the set of $N$ nodes or vertices. A typical sequential exploration algorithm, usually referred to as \emph{``greedy algorithm''} \footnote{called greedy although there is no policy to choose the optimal vertex in each step, see for instance the definition of an unweighted greedy algorithm in \cite{Jungnickel}.} works as follows. Initially, all the vertices are declared as \emph{unexplored}. At each step, it selects a vertex and changes its state into \emph{active}. After this, it takes all of its unexplored neighbors and changes their states into \emph{blocked}. The \emph{active} and \emph{blocked} vertices are considered as \emph{explored} and removed from the set of \emph{unexplored} vertices. The algorithm keeps repeating this procedure until step $T_N^*$, at which all vertices are either active or blocked (or equivalently, the set of the \emph{unexplored} vertex is empty).  Observe that at any step $k$, the active vertices form an independent set (i.e. there are no edges between the nodes of this set) and that  $T_N^*$ is the size of the independent set constructed by the algorithm. Let $Z_k^N$  be the number of explored nodes at time $k$, then $Z_{T_N^*}^{N}=N$. 

Our motivation to study such an exploration process on random graphs is twofold.
On the one hand, exploration processes have received a great amount of attention in spatial structures. It has been considered on discrete structures like $\mathbb Z^d$ (see \cite{Ritchie,ferrari-fernandez-garcia-02}) and point processes (see \cite{Penrose, Baccelli-12}). 
In physics and biological sciences, where it is usually referred to as random sequential absorption, it models phenomena of deposition of colloidal particles or proteins on surfaces (see \cite{Evans1993}).  In communication sciences and wireless networks in particular,
it allows to represent the number of connections for CSMA-like algorithms in a given time-slot,
for a given spatial configuration of terminals (see \cite{Kleinrock} for a classical reference on the protocol definition). 

On the other hand, these dynamics are the simplest procedure to construct (maximal) independent sets and have been extensively studied for specific graphs.  Explicit results for the size of these sets have been obtained for regular graphs in \cite{wormaldDF}, exploiting their particular structure; see also \cite{Gamarnik} for graphs with large girths, and \cite{Bermolen2017} for more general configuration models. In this context, the greedy algorithm is the simplest instance of a local algorithm, i.e., an algorithm using only local information available at each vertex and using some randomness. Recently, it was proven in \cite{Gamarnik} that contrary to previously stated conjectures (for instance, in \cite{Hatami}), local algorithms can not discover asymptotically maximum independent sets (independent set of the maximum size) and stay sub-optimal, up to multiplicative constant, for regular graphs with large girth. Hence, it is natural to look at related questions for Erd\"os-R\'enyi (ER) graphs: we focus on giving estimates of reaching a given size of maximum independent sets by studying the large deviations of the exploration process. 

Thanks to the great amount of independence and symmetry of the edges' collection in a sparse ER graph $G(N, c/N)$, the greedy exploration algorithm is characterized by $\left\{Z^N_k\right\}_k$, a simple one-dimensional Markov process. Consequently, a functional law of large numbers described by a differential equation can be employed to get the macroscopic size of the constructed independent set when the number of nodes goes to infinity (see \cite{Bermolen2017bounds} and references in \cite{McDiarmid}). Diffusion approximations for the process and central limit theorem derived from it for the size $T_N^*$ of the associated independent set are also known, see \cite{Bermolen2017bounds}. Moreover, in \cite{Pittel}, exponential bounds are proved for the probability that the stopping times $t_f(G(N, p/N))$ of the $f$-driven algorithms (in particular, $T_N^*$) belong to certain intervals.
However, to the best of our knowledge, there is no characterization of a large deviation principle (LDP) for both the discrete-time Markov process $\left\{Z_k^N\right\}_k$ and the random variable $T_N^*$, which can give various types of useful information both on the greedy exploration and on the independent set landscape. For example, it allows determining the most probable trajectory for which the independent set's size is bigger/samaller than selected bounds. The present paper's topic is a refined analysis of this simple algorithm by studying the large deviations (LD) for the sequence of processes $\left\{Z_k^N\right\}_k$. As a corollary, we obtain an LDP for the size of the independent set constructed by the algorithm.

Although $\left\{Z_k^N\right\}_k$ is a simple Markov process, as far as we know, computing its LDP does not directly follow from classical results. Indeed, the well-known work of \cite{F&W} is not directly applicable to our process since both the drift and the jump measure involved in the underlying stochastic differential equation depend on the scaling parameter. An LD upper bound for a general family of processes, including processes whose (discontinuous) drift and jump measure depends on the scaling, is presented in \cite{Dupuis_discI}. However, the authors do not provide sufficient conditions to ensure that the general upper bound obtained for simpler processes is still valid for this case.

In this article, we use techniques from the theory of viscosity solutions to Hamilton-Jacobi equations and prove that its LD upper bound not only works for a continuous-time version of $\left\{Z_k^N\right\}_k$, but is also effectively the LD rate function. To prove this LDP, we use the general strategy to study of large deviations of processes proposed by \cite{F&K}, which is based on the convergence of non-linear semigroups.

In general, there are at least two approaches in the literature to prove an LDP. The traditional approach to LDP is via the so-called \emph{change of measure method}. Indeed, beginning with the work of \cite{Cramer} and including the fundamental work on large deviations for stochastic processes by \cite{F&W} and \cite{DonVar}, much of the analysis has been based on a change of measure techniques. In this approach, a tilted or reference measure is identified under which the events of interest have a high probability. The probability of the event under the original measure is bounded in terms of the Radon-Nikodym density that relates both measures. In our case, finding a direct change of measure turns out to be a highly non-trivial task due to the transitions rate dependence on the state and the intricate overall dependence on the scaling parameter.

Another approach is analogous to the Prohorov compactness approach to weak convergence of probability measures (by studying these measures' tightness). It is sometimes referred to as the \emph{exponential tightness method}. This has been established by \cite{Puhalskii}, \cite{OV}, \cite{deAcosta}, \cite{Dupuis}, \cite{Fleming}, \cite{Evans_Ishii}, and others.

The remarkable work of \cite{F&K} consists of combining the tools of probability, analysis, and control theory used in the works of \cite{deAcosta}, \cite{Dupuis}, \cite{Evans_Ishii}, \cite{Fleming78}, \cite{Fleming}, \cite{Fleming99}, \cite{Puhalskii}, and others to propose a general strategy for the study of large deviations of processes. In the case of Markov processes, this program is carried out in four steps:
The first step consists of proving the convergence of non-linear generators $H_N$ and derive the limit operator $\mathbf{H}$. The second step consists of verifying the \emph{exponential compact containment condition}. The third step consists of proving that $\mathbf{H}$ generates a semigroup $\mathbf{V}= \left\{V_t\right\}_t$. This issue is nontrivial and follows, for example, by showing that the Hamilton-Jacobi equation $
f(x)-\beta H \left(x,\nabla f(x)\right)-h(x)=0$ has a unique solution $f$ for all $h\in C(E)$ and $\beta>0$ in a viscosity sense when $\mathbf{H}(f)(x)=H\left(x, \nabla f(x)\right)$. The rate function is constructed in terms of that limit $\mathbf{V}$. This limiting semigroup usually admits a variational form known as the \emph{Nisio semigroup} in control theory. Then, the fourth step consists of constructing a variational representation for the rate function. In a nutshell, as a consequence of the first two steps, the process verifies the exponential tightness condition; the third step assures the existence of an LDP, and the fourth step provides a useful variational version of the rate. 

In our case, after working on the four steps that we mentioned before, we deduce not only a variational form of the rate function but also prove that it can be expressed as an action integral of a cost function $L$. Moreover, by solving the associated Hamilton's equations, the optimization of the rate over a set of trajectories can be transformed into a real parametric function optimization. 

Additionally, the cost function $L$  has a simple interpretation in terms of local deviations for the average of Poisson random variables.  As such, this is a first step to understand how such local algorithms behave on complicated landscapes.

This result also allows us to derive quantitative results about the independent set's size constructed by this algorithm. For instance, we can compute the probability that this size is larger than the asymptotic Erd\"os bound for the maximum independent set when $c\geq 3$ and for the maximum independent set's exact value when $c<e$.
In particular, it sheds light on the relation between the complexity of the landscape and the exploration algorithm. It is known (and coined as the $e$-phenomena in \cite{Spitzer, Jonckheere}) that for $G(N, c/N)$ with $c < e$, an improved local algorithm (the degree-greedy algorithm, which is an improvement of the modification of the greedy algorithm presented in the earlier paper of \cite{Karp-Sipser}) is asymptotically optimal. 
The computation of LD estimates for the greedy exploration (using the asymptotic Erd\"os bound) allows us to give evidence of a phase transition for the independent set landscape around $e$ (we lose some precision here because of using a bound instead of the true asymptotic value of the independent set), but it hints at an interesting connection between complexity phase transitions and explicit large deviations results.

\bigskip

The rest of the paper is organized as follows. 
In Section \ref{sec:main_results}, we define our model and present the main result of this article: a path-state LDP for the greedy exploration process. As a corollary, we obtain an LDP for the size of the independent set discovered by the algorithm and analyze its implications. In Section \ref{sec:LDP_proof}, we briefly describe Feng and Kurtz's theory in our context and prove our main theorem.

\section{Main Results}
\label{sec:main_results}
In this section, we define our process and state the main results of the paper. The key steps of the proof of Theorem \ref{LDP_for_Y} are presented in the next section. 

\subsection{Greedy exploration algorithm}
Let $G\left(N, \frac{c}{N}\right)$ be a sparse Erd\"os-R\'enyi graph for which $V$ is the set of $N$ vertices.  At any step $k = 0, 1, 2, \ldots$, we consider that each vertex is either \emph{active}, \emph{blocked}, or \emph{unexplored}. Accordingly, the set of vertices will be split into three components: the set of active vertices $\mathcal{A}_k$, the set of blocked vertices $\mathcal{B}_k$, and the set of unexplored vertices $\mathcal{U}_k$.

The greedy exploration algorithm in discrete time on a graph $G$ can be described as follows. Initially, it sets $\mathcal{U}_0 = V$, $\mathcal{A}_0 = \emptyset$ and $\mathcal{B}_0 = \emptyset$. To explore the graph, at the $(k+1)$-th step it selects uniformly a vertex $i_{k+1}\in\mathcal{U}_k$ and changes its state into active. After this, it takes all of its unexplored neighbors, i.e. the set $\mathcal{N}_{i_{k+1}} = \{ w \in \mathcal{U}_k | i_{k+1} \text{ shares and edge with } w \}$, and changes their states into blocked. This means that the resulting set of vertices will be given by $\mathcal{U}_{k+1} = \mathcal{U}_k \backslash \{ i_{k+1} \cup \mathcal{N}_{i_{k+1}} \}$, $\mathcal{A}_{k+1} = \mathcal{A}_{k} \cup \{ i_{k+1} \}$ and $\mathcal{B}_{k+1} = \mathcal{B}_k \cup \mathcal{N}_{i_{k+1}}$.  The algorithm iterates this procedure until the step $T_N^*$ at which all vertices are either active or blocked (or equivalently $\mathcal{U}_{T_N^*} = \emptyset$).  Observe that at any step $k$, the active vertices form an independent set and that $\mathcal{A}_{T_N^*}$ is a maximal independent set (because each of the vertices in $V\setminus A_{T_N^*}$ is a neighbour of at least one vertice of $A_{T_N^*}$).

Let $Z_k^N = \left|\mathcal{A}_{k}^{N}\cup \mathcal{B}_{k}^{N}\right|$  be the number of explored vertices at step $k$. By construction, 
$
Z_{k+1}^N = Z_k^N + 1 + \zeta_{k+1}^N,
$
where $\zeta_{k+1}^N$ is the number of unexplored neighbors of the selected active vertex at step $k+1$. The distribution of $\zeta_{k+1}^N$ depends only on the number of already explored vertices $Z_k^N$, that is the distribution is Binomial with updated parameter $N-Z_k^N-1$ and the same edge probability $c/N$. The process $\left\{Z_k^N\right\}_k$ is then a discrete time Markov chain with state space $\left\{0,1,2,...,N\right\}$, increasing, time-homogeneous and with an absorbing state $N$. We are interested in $T_N^*\in \left\{0,1,2,...,N\right\}$, the time at which $\left\{Z_k^N\right\}_k$ reaches $N$, since $T_N^*$ coincides with the size of the maximal independent set constructed by this algorithm.

We use the notation in the work of \cite{F&K} for the discrete time Markov processes case. Let $\tilde{Y}^N = \left\{\tilde{Y}^N_k\right\}_{k\geq 0}$ be a scaled version of the described process: $\tilde{Y}^N_k= \frac{Z_k^N}{N}$. The transition operator of the process $\tilde{Y}^N$ for $x\in E^N=\left\{\frac{k}{N}:\, k=0,1,...,N \right\}$ is:
\begin{equation}\label{eq:TN}
T_N\left(f\right)(x):=T_{\tilde{Y}^N}\left(f\right)(x)=\mathbb{E}\left[f\left(x+\frac{1}{N}+\frac{1}{N}\zeta_{N,x}\right)\right],
\end{equation}
where $\zeta_{N,x}$ is the number of unexplored neighbors of the selected active vertex given that there are already $Nx$ explored vertices. Then $\zeta_{N,x}$ has a Binomial distribution with parameters $N-Nx-1$ and $\frac{c}{N}$. We consider the embedding maps $\eta_N: E^N\rightarrow E$, where $E=[0,1]$. Define the following continuous process:
\begin{gather}
Y^N_t=\tilde{Y}^N_{\left[ Nt \right]}= \frac{Z_{\left[Nt\right]}}{N} \text{ if } t \in \left[0,1\right].
\label{eq:Zcont}
\end{gather}
This process is a semimartingale; moreover, it can be decomposed as
\begin{gather*}
Y_t^N = \int_{0}^{t}\left[1+c\left(1-Y_s^N-\frac{1}{N}\right)\right]\text{d}s+\frac{M_{tN}^N}{N}, 
\end{gather*}
where $\left\{M_t^N\right\}_t$ is a $\mathbb{F}^N=\left\{ \mathcal{F}_t^N\right\}_t$ martingale with $\mathcal{F}_t^N = \sigma\left(Z_{\left[Ns\right]}^N: \, 0\leq s\leq t\right)$.

In \cite{Bermolen2017bounds} it is proved that the sequence of processes $\left\{Y^N\right\}_N$, contained in the space of c\`adl\`ag functions $D_E \left[0,1\right]$, converges in the Skorohod topology to $\left\{z(t) \wedge 1\right\}_{0\leq t\leq 1}$, where $z$ is the solution of the ODE:
\begin{equation}
\dot{z}=1+c\left(1-z\right); \quad z(0)=0.
\label{eq:fluid-limit}
\end{equation}
This equation has an explicit solution given by  $z(t)=\frac{1+c}{c}\left(1-e^{-ct}\right)$. Moreover, a law of large numbers can be deduced for the proportion of vertices that form the independent set constructed by the algorithm. In particular, it is proved that  $\frac{T_N^*}{N}$ converges in probability to $T^*$ defined by $z\left(T^*\right)=1$, i.e. $T^*=\frac{1}{c}\log\left(1+c\right)$.

In the same paper \citep{Bermolen2017bounds} and for a different scaling of the process, a diffusion result is also proved from which a central limit theorem for $\frac{T_N^*}{N}$ is deduced: $\sqrt{N}\left(\frac{T_N^*}{N}-T^*\right)$ converges in distribution to a centered normal random variable with variance $\sigma^2= \frac{c}{2\left(c+1\right)^2}$. Now in the present document, we study an LDP for both the sequence of processes $\left\{Y^N\right\}_N$ and for $\left\{\frac{T_N^*}{N}\right\}_N$. It is known that the results of the central limit theorems and large deviations types are independent of each other, and neither is stronger than the other.  However, we will see that an LDP also automatically provides results of the law of large numbers type.

\subsection{Large Deviation Principle}
This paper aims to present a more refined analysis of the simple exploration algorithm presented in the previous section. As a corollary, in the next section, we deduce an LDP for the sequence of random variables $\left\{ \frac{T_N^*}{N} \right\}_N$.

\begin{theorem}[LDP for $\left\{Y^N\right\}_N$] \label{LDP_for_Y} 
The sequence $\left\{ Y^N \right\}_N$  with $Y^N = \left\{Y^N_t\right\}_{0\leq t \leq 1}$, where $Y^N_t=\frac{Z^N_{\left[Nt\right]}}{N}$, verifies an LDP on $D_E\left[0,1\right]$ with good rate function $I: D_E [0,1] \rightarrow \left[0, +\infty\right]$ such that:
\begin{equation}
I(\varphi) = \begin{cases} 
\intop_{0}^{1} L \left( \varphi, \dot{\varphi} \right) \text{d}t & \text{ if } \varphi \in \mathcal{H}_L,\\
+ \infty & \text{ in other case,}
\end{cases}
\label{eq:rateLDP_Y}
\end{equation}
where $E=\left[0,1\right]$, $L:E\times \mathbb{R}\rightarrow \mathbb{R}$ is the cost function given by
\begin{equation}\label{eq_Lcost_Y}
L(x, \beta)= 
\begin{cases}
(\beta-1) \left[ \log\left( \frac{\beta -1}{c(1-x)}\right) -1 \right] + c(1-x), & \text{ if } x<1 \text{ and } \beta>1,\\
c(1-x), 	& \text{ if } x<1  \text{ and } \beta=1,\\
0, & \text{ if } x=1 \text{ and } \beta=0, \\
+\infty & \text{ in other cases },
\end{cases}
\end{equation}
and $\mathcal{H}_L$ is the set of all absolutely continuous\footnote{A function $\varphi:[a,b]\rightarrow \mathbb{R}$ is absolutely continuous if it can be written as an integral function; i.e. there exists a Lebesgue integrable function $\psi$ on $[a,b]$ such that $\varphi(x)=\varphi(a)+\intop_{a}^{x}\psi(t)dt$ for all $x\in [a,b]$.} function $\varphi:[0,1]\rightarrow [0,1]$ with value $0$ at $0$ and such that the integral $\int_0^1 L\left(\varphi(t), \dot{\varphi}(t)\right)dt$ exists and it is finite.
\end{theorem}

The proof is deferred to Section \ref{sec:LDP_proof}.
\begin{remark}[Law of large numbers]
The cost function \eqref{eq_Lcost_Y} is the Legendre transform w.r.t the second variable of the function 
$H\colon E\times \R \rightarrow \R$
given by
\begin{equation}
H\left(x, \alpha\right)= 
\begin{cases} 
\alpha + c\left(1-x\right)\left(e^{\alpha}-1\right), &\text{ if } 0\leq x< 1, \\ 
0, & \text{ if } x=1, 
\end{cases}
\label{eq:H}
\end{equation}
that is $L \left(x, \beta \right) = \underset{\alpha\in\R}{\sup} \{ \alpha \beta - H\left(x, \alpha \right)\}$.
Since $H\left(x, \alpha\right)$ is convex with respect to $\alpha$, the function $L$ is also convex with respect to $\beta$ and verifies $H \left(x, \alpha \right) = \underset{\beta\in\R}{\sup} \{ \alpha \beta - L\left(x, \beta \right)\}$. We use the notation $H \leftrightarrow L$ for short.
As $L\left(x, \beta \right)=0$ if and only if $\beta = H_{\alpha} \left(x,0\right)$, where $H_{\alpha}\left(x, \alpha\right)$ is the partial derivative of $H\left(x, \alpha\right)$ w.r.t. $\alpha$, the trajectories with zero cost are the ones that verify 
$\dot{\varphi}=H_{\alpha}\left(\varphi,0\right) = 1+c(1-\varphi(t))$. For the initial condition $\varphi(0)=0$, as expected, the unique trajectory that has zero cost is the fluid limit $z$ given by Equation \eqref{eq:fluid-limit} i.e. $I(z)=0$ and $I(\varphi)>0$ for all $\varphi \neq z$.
\end{remark}

The following proposition gives an intuitive interpretation of the cost function $L(x, \beta)$ in terms of the rate function for the average of independent Poisson random variables. 

\begin{proposition}
For $x<1$ and $\beta > 1$, it is verified that
$
L\left(x, \beta\right)= \Lambda_{c(1-x)}^*\left(\beta-1\right),
$
where $\Lambda_{\lambda}^*(u)$ is the LD rate function for the average of independent Poisson random variables with parameter $\lambda$.
\end{proposition}

\begin{proof}
The rate function given by Cr\'amer's theorem for the average of independent random variables Poisson with parameter $\lambda$  is 
$
\Lambda_{\lambda}^*(u)= u \left( \log\left(\frac{u}{\lambda}\right)-1\right)+\lambda
$  
(see \cite{Dembo} for example). To complete the proof it is enough to observe that $L(x, \beta)$ coincides with $\Lambda_{\lambda}^*(u)$ when $\lambda=c(1-x)$ and $u=\beta-1$. 
\end{proof}

The previous result can be explained using the following heuristics (which, of course, are far from a proof but give some intuition):
\begin{itemize}
\item
The graph's sparsity implies that the graph is locally tree-like and that the exploration does not see neighbors of a given vertex being neighbors between them. 

\item
The asymptotic distribution of the number of unexplored neighbors of the selected active vertex is Poisson with a time-varying mean.
In other words, the exploration does not change the Poisson nature of the degree distribution, which can be explained by the fact that the biased size distribution of Poisson distribution is again Poisson.
\end{itemize}

More precisely, the cost of a given curve $x(t)$ such that $x\in \mathcal{H}_L$ with $\dot{x}(t)>1$ for all $t\in [0,1]$ is given by $L\left(x(t), \dot{x}(t)\right)=\Lambda_{\lambda(t)}^*\left(\dot{x}(t)-1\right)$, with $\lambda(t)=c\left(1-x(t)\right)$. For a fixed $t\in(0,1)$, the curve $x(t)$ represents the macroscopic proportion of explored vertices at time $t$. Then, the infinitesimal increment $\dot{x}(t) \approx \frac{x(t+h)-x(t)}{h}$ corresponds to the mean number of new explored nodes in one step (the new active node and its unexplored blocked neighbors), that is:
\begin{gather*}
\frac{Y^N_{t+h}-Y^N_t}{h} 
= \frac{1}{Nh} \sum_{k=[Nt]+1}^{[Nt+Nh]} \left(1+\zeta_k^N\right) \approx 1+ \frac{1}{Nh} \sum_{k=[Nt]+1}^{[Nt+Nh]} \zeta_k^N,
\end{gather*}
where $\zeta_k^N$ has a Binomial distribution with parameters $N- Z_k -1$ and $\frac{c}{N}$. For large values of $N$ and $k\in [[Nt]+1,[Nt+Nh]]$, if $\frac{Z_k}{N}$ is close to $x(t)$, then $\zeta_k^N$ can be approximated by a Poisson random variable with parameter  $(N-Z_k-1)\frac{c}{N} \approx c(1-x(t))$. Observe that, in particular, the mean macroscopic behavior $z(t)$ should verify $\dot{z}(t)= 1+c(1-z(t))$, which is the fluid limit we have already seen. Moreover, the global cost of a deviation from a trajectory $x(t)$ can be interpreted as a consequence of the accumulated cost of microscopic deviations of the average of Poisson random variables of parameter $c(1-x(t))$.

\subsubsection{Rare event probability estimation.}
We now use the previous theorem to estimate probabilities of rare events related to $\left\{Y^N\right\}_N$. In the next section, we apply these results to derive an LDP for the size of the independent set constructed by the algorithm. 

As a consequence of Theorem \ref{LDP_for_Y}, if $A\subset D_E[0,1]$ is a \emph{good set} for $I$ (or an $I$-\emph{continuous set}, see \citep{Dembo}), then
$
\underset{N}{\lim}\frac{1}{N} \log \mathbb{P}\left(Y^N\in A\right)=-\underset{\varphi \in A}{\inf}I(\varphi)
$
. The next proposition will facilitate the computation of this infimum for the sets $A$ of interest.

\begin{proposition}[Rate function optimization]\label{prop:rate_function_optimization} 
\begin{enumerate}
\item The optimization problem for the rate over a set of trajectories $A \subset D_E[0,1]$ can be reduced  to a one-dimensional optimization problem:
$
\underset{\varphi \in A}{\inf} I(\varphi) = \underset{\{\alpha_0\in\R:\, \hat{x}_{\alpha_0}\in \bar{A}\}}{\inf} F\left(\alpha_0\right),
$
where the closure of $A$ is considered with respect to the Skorohod topology, 
\begin{equation}\label{eq:F_alpha_0}
F\left(\alpha_0\right) =\int_{0}^{T_{\alpha_0}} L\left( x_{\alpha_0}(t), \dot{x}_{\alpha_0}(t)\right)\text{d}t,
\end{equation}
$x_{\alpha_0}$ is the solution of the ODE:
\begin{equation}\label{eq:Ham_disc}
\begin{cases}
\dot{x}=1+c(1-x)e^{\alpha},& x(0)=0,\\
\dot{\alpha}= c(e^{\alpha}-1),& \alpha(0)=\alpha_0,
\end{cases}
\end{equation}
$T_{\alpha_0}=\inf\{t\in [0,1]: \, x_{\alpha_0}(t) \geq 1 \}$ and $\hat{x}_{\alpha_0}(t)= x_{\alpha_0}(t) \wedge 1$.\\

\item The explicit solution of Equation \eqref{eq:Ham_disc} is the fluid limit \eqref{eq:fluid-limit} when $\alpha_0=0$. For $\alpha\neq 0$ it is given by:
\begin{equation} \label{eq:sol_Ham}
x_{\alpha_0} (t) = \left[ \frac{1}{ck_0} \log \left(\frac{1-k_0}{1-k_0 e^{ct}} \right) + \frac{1}{e^{-ct}-k_0} -\frac{1}{1-k_0}\right] \left(e^{-ct}-k_0 \right),
\end{equation}
where $k_0 = 1-e^{-\alpha_0}$. In this case, $F\left(\alpha_0\right)$ can be written as 
$
F\left(\alpha_0\right) = \int_{0}^{T_{\alpha_0}} c \left(1-x_{\alpha_0}(t)\right) \left[e^{\alpha(t)}\left(\alpha(t)-1\right) + 1 \right]\text{d}t,
$
where $\alpha(t)= -\log \left(1-k_0e^{ct}\right)$.
\end{enumerate}
\end{proposition}
\bigskip
Then, in other words, Theorem \ref{LDP_for_Y} and the previous proposition ensure that, given that the process $Y^N \in A$, one might expect that
$
\underset{t\in [0,1]}{\sup}\left\vert Y^N_t - \hat{x}_{\alpha_0^*}(t)\right\vert  \approx 0
$ 
for some $\alpha_0^*$ such that $\hat{x}_{\alpha_0^*} \in \bar{A}$.  

\begin{proof} 
To prove the first statement, note that if $\varphi \in \mathcal{H}_L$ is such that $\varphi(t)=1$ for all $t\geq t_0$, then $I(\varphi)=\intop_{0}^{1}L(\varphi, \dot{\varphi})\text{d}t= \int_{0}^{t_0}L(\varphi, \dot{\varphi})\text{d}t$, so just consider the Euler-Lagrange (EL) equation \eqref{eq:EL} for $x<1$ and $\beta>1$. Equation  \eqref{eq:EL} gives conditions for a function $\varphi$ to be a \emph{stationary curve} of the functional $I$:
\begin{equation} \label{eq:EL}
L_x(\varphi, \dot{\varphi})-\frac{\text{d}}{\text{d}t}L_{\beta}(\varphi, \dot{\varphi})=0 \qquad \text{(Euler-Lagrange),}
\end{equation}
where $L_x$ and $L_{\beta}$ are the partial derivatives of $L$ w.r.t. $x$ and $\beta$ respectively. In this case, the path $\{x(t)\}_t$ is a stationary curve of $I$ if it satisfies the following ODE:
\begin{equation}\label{eq:EL_disc}
\begin{cases}
(x-1)\ddot{x}+\left(cx-(1+c)\right)\dot{x}-cx+(1+c)=0,\\
x(0)=0,\\
\dot{x}(0)=v_0.
\end{cases}
\end{equation}
To solve \eqref{eq:EL_disc}, we consider Hamilton's equations, which are equivalent to EL (see \cite{Arnold}, for example):
\begin{equation}\label{eq:Ham}
\begin{cases}
\dot{x}= H_{\alpha}(x, \alpha),\\
\dot{\alpha}=-H_x(x, \alpha),
\end{cases}
\qquad \text{(Hamilton),}
\end{equation}
where $\alpha$ is an auxiliary function. $H_x$ and $H_{\alpha}$ are the partial derivatives of $H$ w.r.t. $x$ and $\alpha$. In our case these equations give \eqref{eq:Ham_disc}. 
We are interested in solutions $x_{\alpha_0}$ of \eqref{eq:Ham_disc} up to the time they reach the value $1$, then we take $\hat{x}_{\alpha_0}$ as in the proposition and get
$
\underset{\varphi \in A}{\inf} I\left(\varphi\right) = 
\underset{\left\{\alpha_0: \, \hat{x}_{\alpha_0} \in \bar{A} \right\}}{\inf} I\left(\hat{x}_{\alpha_0}\right).
$ 

The uniqueness of the solution of the ODE in \eqref{eq:EL_disc}, ensures that a monotony property with respect to the initial condition $\alpha_0$ holds. 
This implies that $x_{\alpha_0}(t)>t$ for all $t$ if $\alpha_0 > -\infty$, then $T_{\alpha_0} = \inf\left\{ t\in [0,1]: \, x_{\alpha_0}(t) \geq 1\right\} \leq 1$ and $I\left(\hat{x}_{\alpha_0}\right) = F(\alpha_0)$ with $F(\alpha_0)$ defined in \eqref{eq:F_alpha_0}. 
Figure \ref{Fig1} contains the graph of $\hat{x}_{\alpha_0}$ for same value of $\alpha_0 <0$ and $\alpha_0 > 0$ compared with the fluid limit $z \wedge 1$. 
\begin{figure} \label{Fig1}
\begin{center}
\begin{tabular}{c c} 
\includegraphics[scale=0.4]{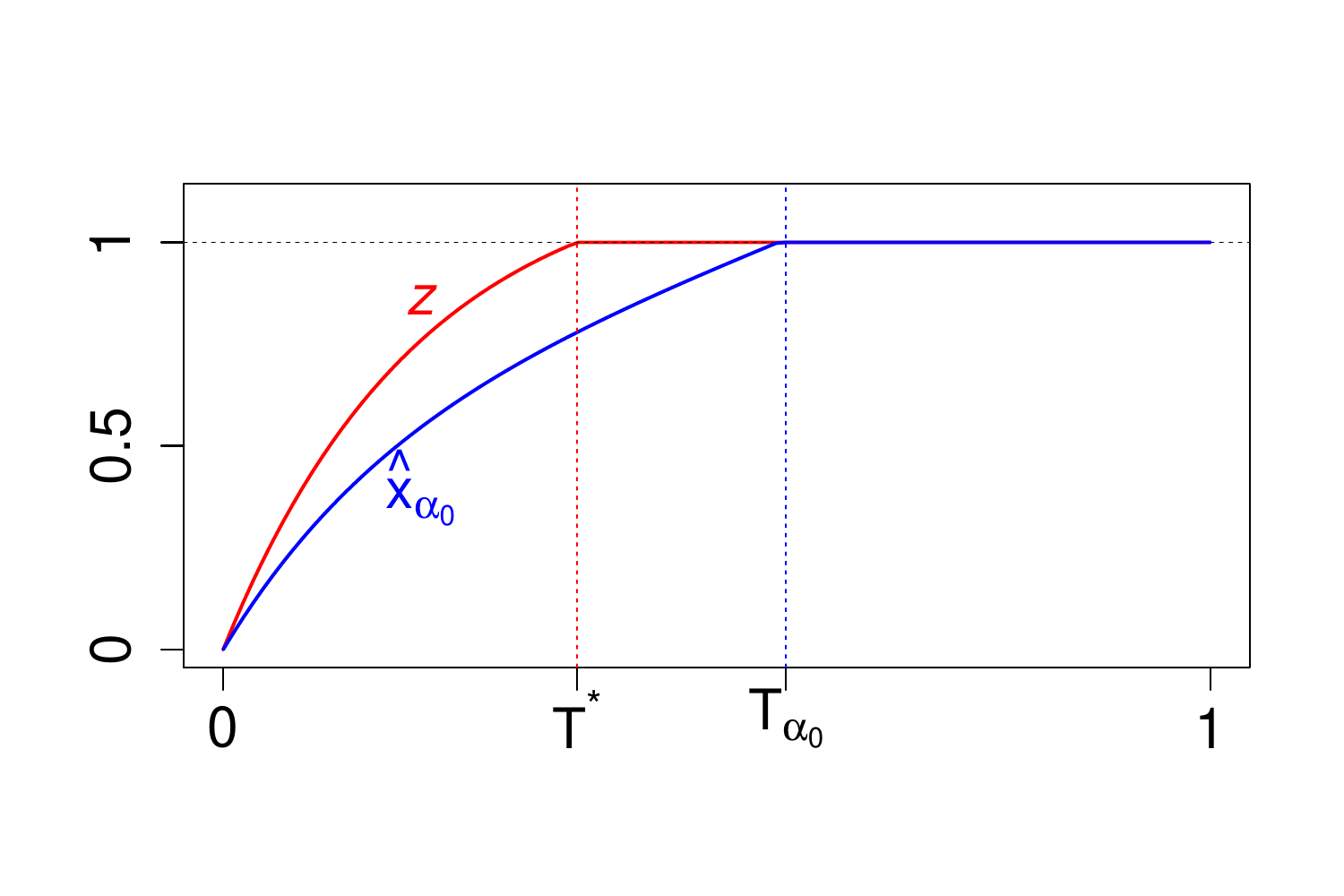} & 
\includegraphics[scale=0.4]{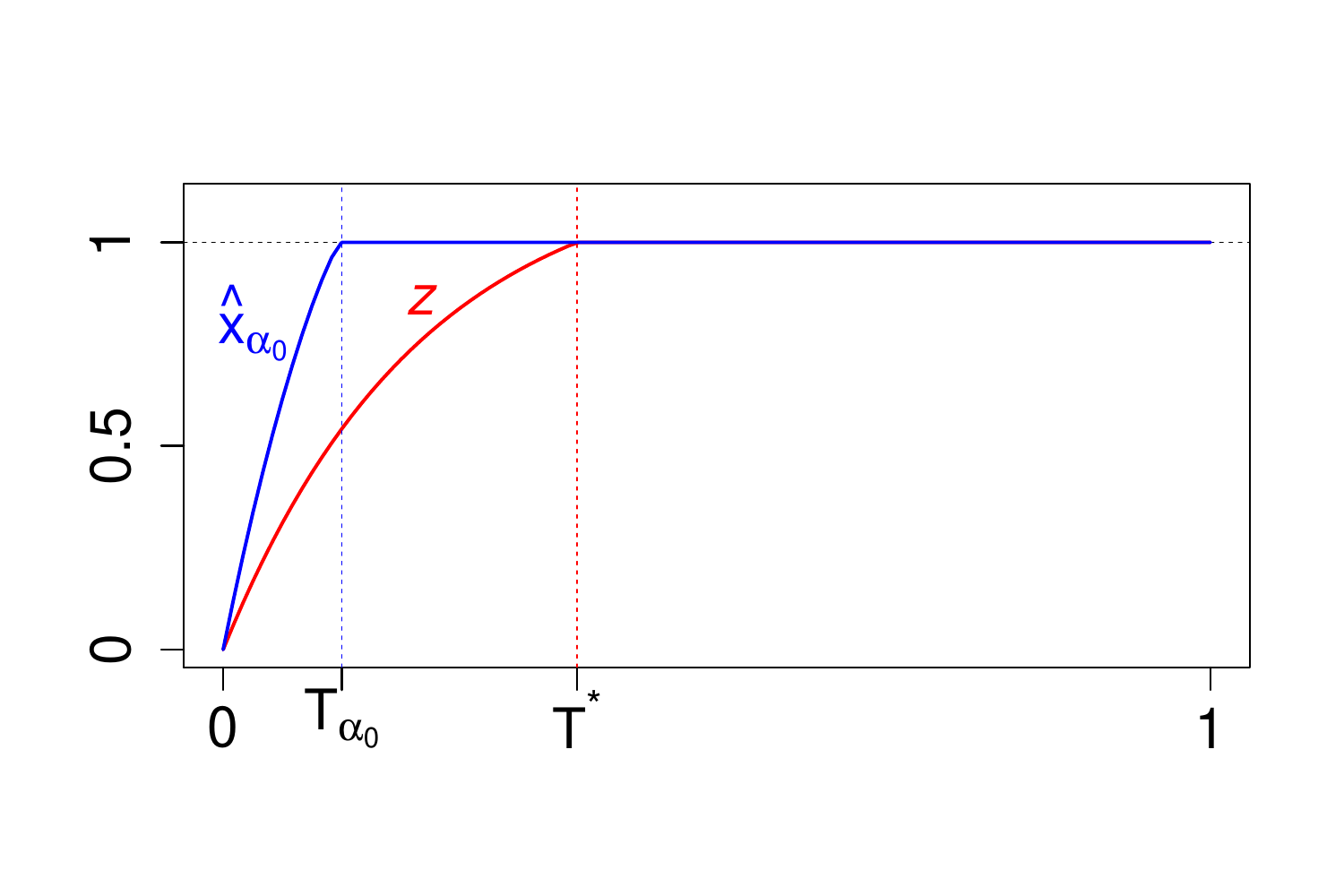} 
\end{tabular}
\caption{Graph of $\hat{x}_{\alpha_0}$ for same value of $\alpha_0 <0$ (left graph) and $\alpha_0 > 0$ (right graph) compared with the fluid limit $z \wedge 1$.}\label{Fig1}
\end{center}
\end{figure}

To prove the second part of the proposition, observe that the fluid limit \eqref{eq:fluid-limit} (until it reaches $x=1$) is a solution of $\dot{z}= 1+ c(1-z)$, so it is a solution of \eqref{eq:Ham_disc} with $\alpha=0$. If $\alpha_0 \neq 0$, the solution $x_{\alpha_0}$ can be found explicitly and it is given by \eqref{eq:sol_Ham}. We use that $x_{\alpha_0}$ is solution of \eqref{eq:Ham_disc} for the simplification of the cost function $L\left(x_{\alpha_0}, \dot{x}_{\alpha_0}\right)$. Figure \ref{fig:F_alpha0} contains the graph of $F\left(\alpha_0\right) = I\left(\hat{x}_{\alpha_0}\right)$ as a function of $\alpha_0$.
\end{proof}
\begin{center}
\begin{figure}
\includegraphics[scale=0.6]{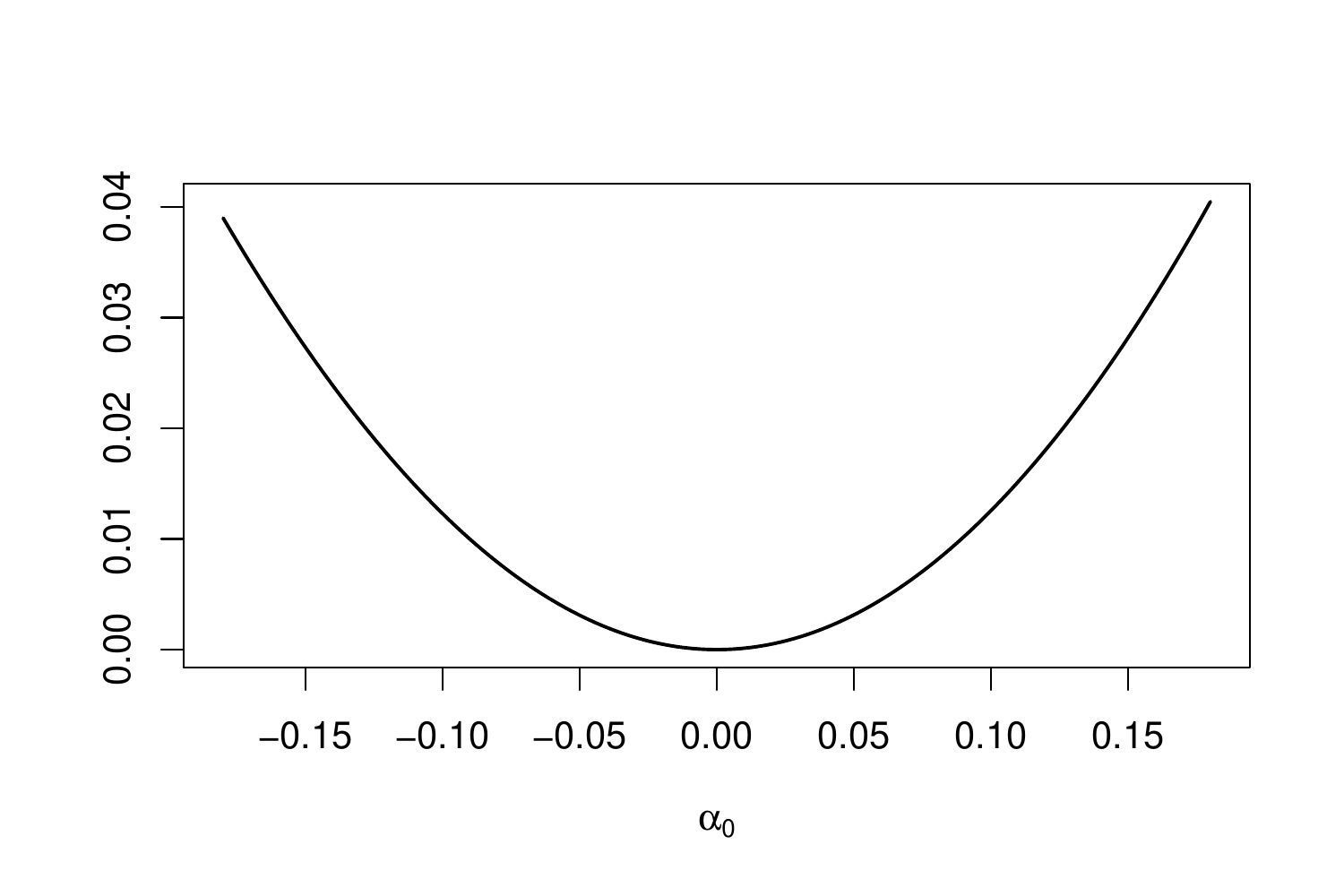} 
\caption{ Graph of the function $F\left(\alpha_0\right) = I\left(\hat{x}_{\alpha_0}\right)$, that is, a parametric version of the rate function. } \label{fig:F_alpha0}
\end{figure}
\end{center}
\begin{remark} Let us introduce some comments on the previous result:
\begin{enumerate}

\item The ODE continuity theorem is verified with respect to the initial condition for the system \eqref{eq:Ham_disc}. Then, the solution $\hat{x}_{\alpha_0}$ with initial conditions $x(0)=0$ and $\alpha(0)=\alpha_0 \approx 0$, is close to the fluid limit $z \wedge 1$.

\item The system \eqref{eq:Ham} is conservative: if $u(t)=\left(x(t), \alpha(t)\right)$ is the solution of (\ref{eq:Ham}) with initial conditions $u_0 = \left(0, \alpha_0\right)$, it verifies $\dot{u}=J\nabla H(u)$ with $J=\begin{pmatrix}0 & 1\\-1 & 0\end{pmatrix}$. Since $J$ is an antisymmetric matrix, it results that $\frac{\text{d}}{\text{d}t}H(u)=\left(\nabla H(u)\right)^t J \nabla H(u)=0$ for all $t$. Then, the solutions of the general equation \eqref{eq:Ham} are contained in the level sets of the Hamiltonian $H$.
\end{enumerate}
\end{remark}

\subsection{LDP for the size of the independent set constructed by the algorithm}
In the previous section, we presented a path-space LDP for the exploration process defined in Section \ref{Introduction}. In this section, we derive from this theorem and the previous proposition about the rate optimization over a specific set, an LDP for the sequence of random variables $\left\{\frac{T_N^*}{N}\right\}_{N}$. This theorem provides quantitative results for the probability of the independent set's size being bigger/smaller than selected bounds. 

\begin{theorem}\label{LDP_T}
Consider $T_N^*$ defined before as the stopping time of the greedy exploration process over $G(N,\frac{c}{N})$. 
\begin{enumerate}
\item If $\varepsilon>0$ is such that $T^*+\varepsilon<1$, then $$\underset{N}{\lim} \frac{1}{N} \log \mathbb{P} \left(\frac{T^*_N}{N} \geq T^*+\varepsilon \right) = -F\left(\alpha_0 (T^* + \varepsilon)\right),$$ where $\alpha_0(T^*+\varepsilon)$  is the unique real number $\alpha_0<0$ such that $T_{\alpha_0} = T^*+\varepsilon$.
\item If $\varepsilon>0$ is such that $T^*-\varepsilon>0$, then $$\underset{N}{\lim} \frac{1}{N} \log \mathbb{P} \left(\frac{T^*_N}{N} \leq T^*-\varepsilon \right) = -F\left(\alpha_0 (T^* - \varepsilon)\right),$$ where $\alpha_0(T^*-\varepsilon)$ is the unique real number $\alpha_0>0$ such that $T_{\alpha_0} = T^*-\varepsilon$.
\end{enumerate}
In both cases $F(\alpha_0)$ and $T_{\alpha_0}$ are as in Proposition \ref{prop:rate_function_optimization}.
\end{theorem}

\begin{proof} We only prove the first statement because the proof of the second one is analogous. Define the set $A_{\varepsilon}$ such that $A_{\varepsilon}=\{\varphi \in D_E\left[0,1\right]: \, \varphi(0)=0, \varphi$ is increasing, $0\leq \varphi(t)\leq 1$  for all $t$ and $\inf\left\{t: \, \varphi(t)=1\right\}\geq T^*+\varepsilon\}$. By construction, $A_{\varepsilon}$ is a good set for $I$, then
\begin{gather*}
\lim_{N} \frac{1}{N} \log \mathbb{P} \left(\frac{T^*_N}{N} \geq T^*+\varepsilon \right)  = \lim_N \frac{1}{N} \log \mathbb{P} \left(Y^N\in A_{\varepsilon} \right)
= -\inf_{\left\{\alpha_0:\, \hat{x}_{\alpha_0}\in A_{\varepsilon}\right\}} F\left(\alpha_0\right).
\end{gather*}

Let $x_{\alpha_0}$ be the solution of the homogenous ODE \eqref{eq:EL_disc} with initial velocity $v_0 = \dot{x}_{\alpha_0}(0)=1+ce^{\alpha_0}$. The uniqueness of the solution ensures that the following monotony property with respect to the initial condition is verified:
\begin{gather*}
\text{ if } \alpha_0 <\alpha_1 \Rightarrow x_{\alpha_0}(t)<x_{\alpha_1}(t) \text{ for all } t \Rightarrow T_{\alpha_0} >T_{\alpha_1}.
\end{gather*}
In addition, it can be seen that for all $T\in\left(T^*,\, 1\right)$, there exists a unique value $\alpha_0= \alpha_0(T)<0$ such that $x_{\alpha_0}(T)=1$ (i.e. $T=T_{\alpha_0}$). Then, there is only one $\alpha_0^*<0$ such that $x_{\alpha_0^*}\left(T^*+\varepsilon\right)=1$ and
\begin{itemize}
\item if $\alpha_0 \leq \alpha_0^* \Rightarrow T_{\alpha_0}\geq T^*+\varepsilon \Rightarrow \hat{x}_{\alpha_0} \in A_{\varepsilon} $,
\item if $\alpha_0 > \alpha_0^* \Rightarrow T_{\alpha_0}< T^*+\varepsilon \Rightarrow \hat{x}_{\alpha_0} \notin A_{\varepsilon} $,
\end{itemize}
which implies that
$
 \underset{\{\alpha_0: \hat{x}_{\alpha_0}\in A_{\varepsilon}\}}{\inf} F\left(\alpha_0\right) = \underset{\{\alpha_0 \leq \alpha_0^*\}}{\inf} F\left(\alpha_0\right)
$.
To complete the proof it suffices to prove that $\underset{\{\alpha_0 \leq \alpha_0^*\}}{\inf} F\left(\alpha_0\right)= F(\alpha_0^*)$.
Let $h(\alpha_0, t)= L(x_{\alpha_0}, \dot{x}_{\alpha_0})$ and $\alpha_1 < \alpha_2 < 0$. Using the monotony that we mentioned before, it can be seen that $\frac{\partial}{\partial \alpha_0} h(\alpha_0,t)<0$ for all $\alpha_0<0$ and $t\in [0,1]$,  that is  $h(\alpha_1,t)>h(\alpha_2,t)$ for all $t$. Finally, since $T_{\alpha_1}>T_{\alpha_2}$ we obtain:
\begin{gather*}
F(\alpha_2)=\int_{0}^{T_{\alpha_2}} L(x_{\alpha_2}, \dot{x}_{\alpha_2})\text{d}t \leq 
\int_{0}^{T_{\alpha_2}} L(x_{\alpha_1}, \dot{x}_{\alpha_1})\text{d}t <
\int_{0}^{T_{\alpha_1}} L(x_{\alpha_1}, \dot{x}_{\alpha_1})\text{d}t= F(\alpha_1),
\end{gather*}
which completes the proof.
\end{proof} 

\subsection{On the size of the maximum independent set}
The problem of finding the \emph{maximum} independent sets in deterministic graphs is known to be NP-hard. An interesting research question is to find classes on random graphs where finding \emph{maximum} independent sets can be (at least at the first order in $N$) obtained with polynomial complexity.
This question is, of course, an instance of a more general viewpoint which
aims at identifying phase transitions in the analysis of combinatorial optimization problems, allowing to describe drastically different sce\-narios depending on a few macroscopic parame\-ters, sometimes called order-parameters.
 
This type of results has been proven to hold for Erd\"os-R\'enyi graphs and configura\-tion models in \cite{Spitzer} and \cite{Jonckheere}.  The order-parameter being $c$ the mean number of neighbors of a given node. Interestingly the phase transition does not correspond for the graph
to be subcritical ($c<1$) but to a much finer property of the landscape of maximal independent sets. The phase transition corresponds to $c < e$ and differentiates between regimes where a simple degree-greedy algorithm reaches (asymptotically) the maximum independent set or not. This same phase transition is reflected in the properties of the spectrum of the graph, see \cite{salez}.

We conjecture that the large deviations characteristics of the greedy algorithm for discovering maximum independent set also have an interesting transition for values of $c$ around $e$. Since the exact optimal order-one asymptotic value of the maximal independent set's size is known only for values of $c<e$, we cannot yet display a full characterization of this phenomena.
We can, however, obtain interesting numerical results by using the 
Erd\"os bound, instead of the true value. Let $\sigma_N$ the \emph{maximum} size of the independent set of an ER graph $G(N,c/N)$, then a.s.
$\sigma_N \leq \frac{2 \log(c)}{c} N (1 +o(N))$ if $c\geq 3$.
In Figure \ref{figure:F_sigma} we compute the large deviation rate corresponding to the event $\{ \frac{T_N^*}{N} \geq \sigma_i^*(c) \}$ for $i=1,2$. Here $\sigma_1^*$
is the exact proportion of the maximum independent set of an ER graph $G(N, c/N)$ when $c<e$ (\cite{Jonckheere}) and it is given by $\sigma_1^*(c)=w(c)+\frac{c}{2}\left(w(c)\right)^2$ with $w(c)=e^{-W(c)}$ and $W(x)$ the Lambert function.  The value $\sigma_2^*(c)=\frac{2}{c}\log(c)$ is the Erd\"os upper bound for the proportion of the maximum independent set for $c\geq 3$. 

Though the numerical computations for $c>e$ could give largely overestimated values,
we believe it nevertheless illustrates the clear change of regime around the value $e$. It shows that the independent sets geometry changes, leading to significantly greater large deviations constants for the greedy exploration when $c$ gets larger than $e$.
This characterization of the ``energy" landscape is a usual situation in statistical physics where interesting phase transitions can be well described through large deviations, see \cite{Touchette}.

\begin{figure}
\begin{center}
\includegraphics[scale=0.5]{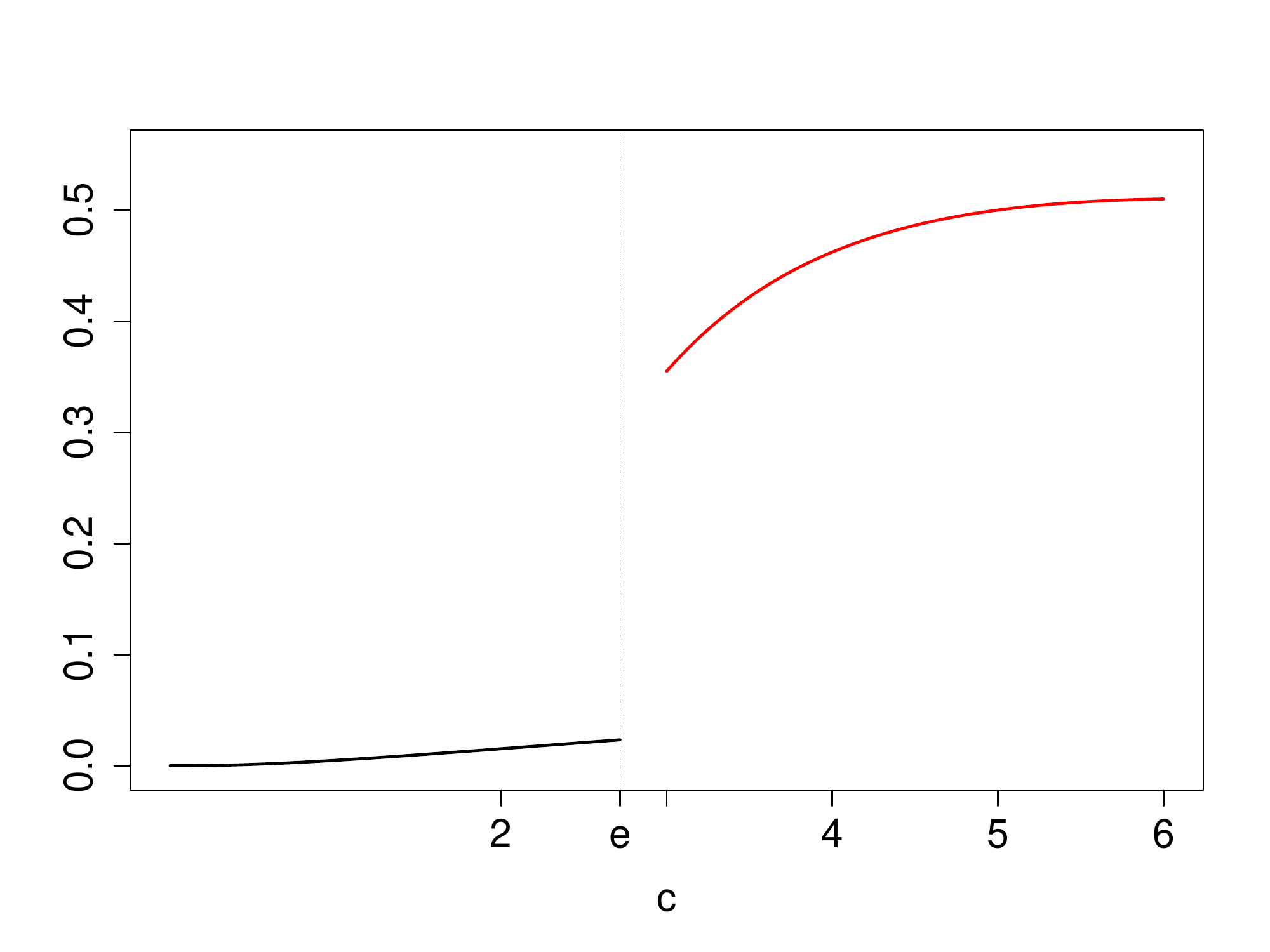} 
\end{center}
\caption{Evolution of $F(\alpha_0(\sigma_1^*(c)))$ for $0<c<e$ and $F(\alpha_0(\sigma_2^*(c)))$ for $c\geq 3$. }
\label{figure:F_sigma}
\end{figure}
\section{Proof of  Theorem \ref{LDP_for_Y}}
\label{sec:LDP_proof}

In this section, we first briefly describe the theory and main results of \cite{F&K} in our context, and then we prove that the previously defined sequence of processes $\{Y_N\}_N$ verifies their assumptions. We organize the main assumptions in four steps described below.

\subsection{Theory of Feng and Kurtz in our context} 
As mentioned in Section \ref{Introduction}, Feng and Kurtz based their study of large deviations on the \emph{exponential tightness} method. The main result according to this approach is Bryc's theorem. This theorem states that if $\left(\mathcal{X}, d\right)$ is a Polish space; $\left\{ \mathbb{P}_N \right\}_N$ is an exponentially tight \footnote{The sequence $\left\{ \mathbb{P}_N \right\}_N$ is exponentially tight if for all $\alpha >0$ exists a compact $K_{\alpha} \subset \mathcal{X}$ such that ${\displaystyle \limsup_{N} \frac{1}{N} \log \mathbb{P}_N \left(K_{\alpha}^{c}\right) \leq -\alpha}$.} sequence of probability measures defined on $\mathcal{X}$, and the following limit exists: 
\begin{gather*}
\Lambda(f) = \lim_{N} \frac{1}{N} \log \int_{\mathcal{X}} e^{Nf(x)}\text{d}\mathbb{P}_N (x) 
\quad \forall f\in C_{b}(\mathcal{X}),
\end{gather*}
then $\left\{ \mathbb{P}_N \right\}_N$ satisfies an LDP with rate function $I: \mathcal{X}\rightarrow \mathbb{R}$ such that 
$
I(x)= \sup \left\{ f(x)-\Lambda(f): \; f\in C_{b}(\mathcal{X}) \right\},
$
where $C_b(\mathcal{X})$ is the space of bounded and continuous functions $f:\mathcal{X}\rightarrow \mathbb{R}$.
\bigskip

Consider now the case in which $\left\{ \mathbb{P}_N \right\}_N$ comes from a sequence of continuous or discrete-time Markov processes $\left\{Y^N\right\}_N$ with space of states $E_N$. Suppose that $E_N \subset E$ for all $N$. Then $\mathcal{X}= D_E \left[0, T\right]$ ($T\leq +\infty$), the space of c\`adl\`ag functions equipped with the Skorohod topology (in the discrete case, the time is  transformed to be continuous, as we did for the process $\left\{Z_k^N\right\}_k$ in \eqref{eq:Zcont}). There are results in the literature that ensure equivalent conditions to the exponential tightness, but the calculation of $\Lambda (f)$ is very difficult or even impossible. The theory of Feng and Kurtz solves both, this problem and the exponential tightness. As the transitions characterize the Markov dynamics, instead of calculating $ \lim\limits_{N} \frac{1}{N} \log \mathbb{E} \left[ e^{N f \left(Y^N \right)} \right]$, the convergence of the \emph{Fleming semigroups} (see \cite{Fleming}) is studied: $V_{t}^{N}: Dom \left(V_{t}^{N}\right) \subset B(E)\rightarrow B(E) $ such that $V_{t}^{N}(f)(x)=  \frac{1}{N} \log \mathbb{E}\left[e^{Nf(Y^{N}_{t})}|Y_{0}^{N}=x\right],$
where $B(E)$ is the space of bounded, Borel measurable functions (i.e. $t$ is fixed and the domain of the functions $f$ is $E$ instead of the much more complex space $D_E \left[0,T\right]$). In \cite{F&K} it is proved that, under certain assumptions, the convergence of the Fleming semigroups ensures an LDP. Actually, instead of studying the convergence of $\left\{ V_t^N\right\}_{t \in \left[0,T\right]}$, the convergence of their (nonlinear) generators $H^N$ is studied. The general idea is that if there is a functional $\mathbf{H}$ such that $H_N \rightarrow \mathbf{H}$ (the type of convergence will depend on each case), $\mathbf{H}$ generates a semigroup $\mathbf{V}= \left\{V_t\right\}_t$ and the \emph{exponential compact containment condition} is verified, then the sequence $\left\{Y^N \right\}_N$ verifies an LDP with rate function $I$ that depends on $\mathbf{V}$. Moreover, if $\mathbf{H}$ is such that $\mathbf{H}\left(f\right)(x) = H \left(x, f'(x)\right)$ for all $f\in C^1\left(E\right)$ and Conditions \textbf{8.9}, \textbf{8.10} and \textbf{8.11} of \cite{F&K} are also verified, we obtain a variational version of $I$. In our particular case, the rate will be written as an action integral of $L\left(x, \beta\right) \leftrightarrow H\left(x, \alpha\right)$.

The main steps of the proof of Theorem \ref{LDP_for_Y} are now briefly outlined. As a consequence of the first two steps, the process $\{Y^N\}_N$ verifies the exponential tightness condition. Step 3 assures an LDP via the comparison principle, and finally, Step 4 provides a useful variational version of the rate. Let $T^N$ the transition operators defined in Equation \ref{eq:TN} and $H_N (f)= \log\left( e^{-Nf}T^N \left(e^{Nf}\right)\right)$.\\

\begin{paragraph}
{\textbf Step 1.}\emph{Verify the convergence of the sequence of operators $H_N$ and derive the limit operator $\mathbf{H}$.}  See Proposition \ref{prop:step1} and note that $\mathbf{H}\left(f\right)(x)=H\left(x, f'(x)\right)$ for $f\in C^1(E)$.
\end{paragraph}\\

\begin{paragraph}
{\textbf Step 2.} \emph{Verify the exponential compact containment condition.} 
The sequence $\left\{Y^N\right\}_N$ verifies the \emph{exponential compact containment condition} if for all $\alpha>0$, there exists $K_{\alpha} \subset E$ compact such that $\underset{N}{\limsup}\frac{1}{N} \log \mathbb{P} \left( \left\{ \exists \, t\in [0,T]: \, Y_{t}^{N} \notin K_{\alpha} \right\} \right) \leq -\alpha$. In our case $E$ is compact, so this condition is trivially verified by taking $K_{\alpha}=E$.
\end{paragraph}\\

\begin{paragraph}
{\textbf Step 3.} \emph{Prove that $\mathbf{H}$ generates a semigroup $\mathbf{V}= \left\{V_t\right\}_t$} (comparison principle). This is the most technical step. By definition, $V_t ^N$ verifies $\frac{\text{d}}{\text{d}t} V_{t}^{N}(f)= H_N \left(V_{t}^{N}(f)\right)$. Then $\mathbf{H}$ generates a semigroup if there is $\mathbf{V}= \left\{V_t\right\}_t$ such that for all $f\in Dom(\mathbf{V})$,
\begin{gather*}
\frac{\text{d}}{\text{d}t} V_t(f)  =  \mathbf{H} \left( V_t(f) \right); \quad
V_0(f) =  f.
\end{gather*}
The theorem of \cite{Cra&Lig} implies that $\mu_N (t)= \left( Id -\frac{t}{N} \mathbf{H}\right)^{-N}$ converges to the solution of the previous equation if $\mathbf{H}$ is m-dissipative. Then we need to prove that for all $h\in C(E)$ and $\beta>0$, there exists $f\in C^1(E)$ such that 
\begin{gather}
f-\beta \mathbf{H}(f)-h=0.
\label{eq:dissipative}
\end{gather}
However the verification of this property can be a formidable obstacle. One way out is to work with viscosity solutions and prove that the \emph{comparison principle} (see Definition \ref{comparison_principle}) for Equation \eqref{eq:dissipative} is verified. If the comparison principle is verified, then the operator $\mathbf{H}$ can be extended to $\hat{\mathbf{H}}$ such that $\hat{\mathbf{H}}$ is m-dissipative and generates a semigroup $\mathbf{V}$ (see Theorem {\bf 8.27} of \cite{F&K}). As mentioned by \cite{F&K}, the verification of the comparison principle is an analytic issue and often gives the impression of being rather involved and disconnected from the probabilistic large deviations problems. An in-depth study of the comparison principle for Hamilton-Jacobi equations in this context is presented in \cite{Kraaij}, using results from \cite{CIL} and Chapter 9 of \cite{F&K}. We follow these ideas to prove the comparison principle in our case. See Proposition \ref{prop:step3}.
\end{paragraph}\\

Once we have verified these three steps, Theorem \textbf{6.14} from \cite{F&K}  assures that $\left\{Y^N \right\}_N$ is exponentially tight and satisfies an LDP with rate function $I$ defined implicitly in terms of $V_t$. This is a theoretical result but does not provide a useful characterization of the rate. The next step provides a simplified version of the rate that can be used in practice.\\

\begin{paragraph}
{\textbf Step 4.} \emph{Construct a variational representation for the rate function $I$}. Let $L\left(x, \beta\right) \leftrightarrow H\left(x, \alpha\right)$. We state the following result:

\begin{theorem}\label{thm:variational_version}
If Conditions \textbf{8.9}, \textbf{8.10} and \textbf{8.11} of \cite{F&K} are also verified, then:
\begin{enumerate}
\item[(a)] $V_t\left(f\right) = \mathcal{V}_t\left(f\right)$ for all $f \in Dom\left(\mathbf{V}\right)$, 
where
\begin{gather*}
\mathcal{V}_t(f)(x_0)= \underset{\{(\mathbf{x},\lambda)\in \mathcal{Y}:\, \mathbf{x}(0)=x_0\}}{\sup} \left\{ f(\mathbf{x}(t)) - \iint_{U\times [0,t]}
L\left(\mathbf{x}(s), u\right) \lambda(\text{d}u \times \text{d}s)\right\},
\end{gather*}
is the Nisio semigroup (see \cite{ Nisio76, Nisio78, Fleming99, ElKaroui}) associated to the cost function $-L$. $\mathcal{Y}$ is a control subset that we define in subsection \ref{step4}.

\item[(b)] $I(\mathbf{x})= \underset{ \{ \lambda: \, (\mathbf{x}, \lambda)\in \mathcal{Y} \} }{\inf} 
\left\{ \iint_{U\times [0,1]} L(\mathbf{x}(s), u) \lambda ( \text{d} u \times \text{d} s ) \right\}$.

\item[(c)] Moreover, the rate function can be written as an action integral:
\begin{gather*}
I(\mathbf{x})= \int_{0}^{1} L(\mathbf{x}(s), \dot{\mathbf{x}}(s))\text{d}s, 
\end{gather*}
if $\mathbf{x}\in \mathcal{H}_L$ and $I(\mathbf(x))=+\infty$ in another case.
\end{enumerate}
\end{theorem}  

\begin{proof}
The first two sentences, (a) and (b), are proved by Theorems \textbf{8.14}, \textbf{8.23}, \textbf{8.27} and \textbf{8.29} of \cite{F&K} taking $E=[0,1]$, $U=\R$, the linear operator $A:C^1(E)\rightarrow M\left(E\times U\right)$ such that $A(f)(x,u)=f'(x)u$, $L(x, \beta) \leftrightarrow H(x, \alpha)$ and $\Gamma = E\times U$. For (c) we use that $L$ is convex w.r.t. the second variable and Jensen's inequality.
\end{proof}
Then, it remains for us to verify Conditions {\bf 8.9}, {\bf 8.10} and {\bf 8.11}. We do this in Section \ref{step4}.
\end{paragraph}\\

We organize the proof of Theorem \ref{LDP_for_Y} using the steps mentioned above, that are presented as propositions. As mentioned before, Step 2 is trivially verified in our case.


\subsection{Step 1: Convergence of the nonlinear operators}

Let $H_N: Dom(H_N)\subset B(E)\rightarrow B(E)$ such that $H_N  = \log \left[ e^{-Nf(x)} T_N \left( e^{Nf}\right)(x) \right]$ with $T_N$ the transition operator for the process $\left\{\frac {Z_k^N}{N}\right\}_k$.

\begin{proposition} \label{prop:step1}
There exists a functional $\mathbf{H}$ such that $H_N$ converges to $\mathbf{H}$ when $N\rightarrow \infty$ in the following sense: $\underset{N\rightarrow \infty}{\lim}      \underset{x\in E^N}{\sup} \left|H_N(f)(x)-\mathbf{H}(f)(x)\right|=0$ for all $f\in C^1(E)$. The functional $\mathbf{H}: C^{1}(E) \rightarrow B(E)$ is such that $ \mathbf{H}(f)(x)=H(x, f'(x))$, where $H: E \times \mathbb{R}\rightarrow \mathbb{R}$ is defined by
\begin{gather}
H(x, \alpha) = \begin{cases} \alpha + c(1-x)\left(e^{\alpha}-1 \right), & \text{ if } x<1,\\ 0, & \text{ if } x=1. \end{cases}
\label{eq:H}
\end{gather}
\end{proposition}

\begin{proof} 
Let us first consider the case where $f\in C^2(E)$. Let $x\in E^N$, $x\neq 1$, and $\zeta_{N,x}$ be the number of unexplored neighbors of the selected vertex, given that there are already $Nx$ explored vertices, then
\begin{gather*}
H^N(f)(x) =
\begin{cases}
\log \mathbb{E}\left[ \exp\left\{\frac{f\left(x+\frac{1}{N}+\frac{\zeta_{N,x}}{N}\right) -f(x)}{\frac{1}{N}}\right\} \right], & \mbox{ if } 0\leq x<1, \\
0, & \mbox{ if }  x=1.
\end{cases}
\end{gather*}
It is enough to prove that
\begin{gather*}
\underset{N\rightarrow \infty}{\lim} \underset{x\in E^N \setminus \{1\}}{\sup} \mathbb{E}\left[ \exp \left\{\frac{f\left(x+\frac{1}{N}+\frac{\zeta_{N,x}}{N}\right) -f(x)}{\frac{1}{N}}\right\}\right] -\mathbb{E} \left[e^{f'(x)\left(1+\zeta_{N,x}\right)}\right] = 0,
\end{gather*}
and this is verified since both $\mathbb{E}\left[ e^{f'(x)(\zeta_{N,x}+1)} \left(e^{ \pm \frac{M_f}{2} \frac{(\zeta_{N,x}+1)^2}{N}}-1 \right) \right]$ converge to zero, being $M_f= \underset{\theta\in[0,1]}{\sup} \left\vert f''(\theta) \right\vert < \infty$. If $x=1$, then $H^N(f)(1)=\mathbf{H}(f)(1)=0$ for all $N$. The result can be extended for  $f\in C^1(E)$ by taking a sequence $\left\{f_m\right\}_m \subset C^2(E)$ such that $\underset{m\rightarrow \infty}{\lim} \underset{x\in E}{\sup} \left|f_m(x)-f(x)\right|=0$ and the triangular inequality.
\end{proof}

\subsection{Step 3: Comparison principle}
As mentioned before, the verification that for all $\beta>0$ and $h\in C(E)$ there exists a solution $f\in C^1(E)$ for the equation $f(x)-\beta H\left(x, f'(x)\right)-h(x)=0$ is difficult or imposible. An alternative is to prove the existence (and uniqueness) of viscosity solutions. Moreover, due to Theorem {\bf 6.14} of \cite{F&K}, it is enough to prove that the \emph{comparison principle} is verified for this Hamilton-Jacobi equation. The ideas to prove it were taken from \cite{Kraaij}, Chapter 9 of \cite{F&K} and \cite{CIL}.
\bigskip

Let $\beta>0$, $h\in C(E)$ and $F_{\beta,h}: E\times \mathbb{R}^2\rightarrow \mathbb{R}$ such that $F_{\beta,h}(x, \epsilon, p)= \epsilon-\beta H(x,p)-h(x)$. Consider the following Hamilton-Jacobi equation:
\begin{equation}\label{eq:HJ}
F_{\beta, h}(x, f(x), f'(x))=0 \, \ \forall x\in E. 
\end{equation}

\begin{definition}\label{comparison_principle} 
The function $\mu \rightarrow \R$ is a (viscosity) subsolution [\emph{supersolution}] of Equation (\ref{eq:HJ}) if it is bounded, upper [\emph{lower}] semi-continuous (u.s.c.) [\emph{l.s.c}] and  for all $\phi\in C^1(E)$ and $x_0\in E$ such that $\mu-\phi$ has a maximum [\emph{minimum}] at $x_0$, we have $F_{\beta,h}\left(x_0, \mu(x_0), \phi'(x_0)\right)\leq 0$ [$\geq 0$]. Equation (\ref{eq:HJ}) verifies the \emph{comparison principle} if for any subsolution $\mu$ and supersolution $v$, it is verified that $\mu \leq v$.
\end{definition}

If the comparison principle is verified, then if there is a viscosity solution (both sub and supersolution), it is unique. In Chapter 9 of \cite{F&K} algorithms are suggested for constructing sequences $x_{\alpha}$, $y_{\alpha}$ (with $\alpha\rightarrow +\infty$) such that $(x_{\alpha}, y_{\alpha})\rightarrow (z,z)$ and $z$ verifies $\mu(z)-v(z)=\underset{x \in E}{\sup} \left\{ \mu(x)- v(x)\right\}$. 

\begin{proposition}\label{prop:step3}
For each $\beta>0$ and $h\in C(E)$ the comparison principle is satisfied for Equation \ref{eq:HJ} with $f\in C^1(E)=Dom(\textbf{H})$.
\label{prop:step3}
\end{proposition}

\begin{proof}
Let $\mu$ be a subsolution and $v$ a supersolution of Equation (\ref{eq:HJ}).\\
Let $\psi:\left[0,1\right]^2 \rightarrow \mathbb{R}^+$ such that $\psi(x,y)= \frac{1}{2}\left(x-y\right)^2$ and let $x_{\alpha}, y_{\alpha}\in E$ such that
\begin{gather*}
\mu(x_{\alpha})-v(y_{\alpha})-\alpha \psi(x_{\alpha}, y_{\alpha}) = \sup_{x,y \in E} \left\{ \mu(x)-v(y)-\alpha \psi(x,y) \right\}.
\end{gather*}
As consequence of Proposition {\bf 4.2} in \cite{Kraaij} it is enough to prove that the following inequality holds:
\begin{gather*} 
\liminf_{\alpha\rightarrow \infty} H\left( x_{\alpha}, \alpha \psi_x\left(x_{\alpha}, y_{\alpha}\right)\right) -  H\left( y_{\alpha}, \alpha \psi_x\left(x_{\alpha}, y_{\alpha}\right)\right) \leq 0,
\end{gather*}
where $\psi_x$ is the derivative of $\psi$ w.r.t. $x$. If $z\in \left[0, 1\right)$, then
\begin{gather*}
H\left( x_{\alpha}, \alpha \psi_x\left(x_{\alpha}, y_{\alpha}\right)\right) -  H\left( y_{\alpha}, \alpha \psi_x\left(x_{\alpha}, y_{\alpha}\right)\right) = -c \left(e^{\alpha\left(x_{\alpha}-y_{\alpha}\right)}-1\right)\left(x_{\alpha}-y_{\alpha}\right).
\end{gather*}
By Proposition {\bf 3.7} in \cite{CIL}, we know that $x_{\alpha}-y_{\alpha}\rightarrow 0$ and due to Lemma {4.5} in \cite{Kraaij} we have:
\begin{gather*}
\sup_{\alpha} H\left( y_{\alpha}, \alpha \left(x_{\alpha}-y_{\alpha}\right)\right) = \sup_{\alpha} 
\alpha \left(x_{\alpha}-y_{\alpha}\right) + c\left(1-y_{\alpha}\right) \left(e^{\alpha \left(x_{\alpha}-y_{\alpha}\right) }-1\right) <\infty,
\end{gather*}
and this implies that $\underset{\alpha} {\sup} \, \alpha \left(x_{\alpha}-y_{\alpha}\right) <\infty$. Then $\left\{\alpha \left(x_{\alpha}-y_{\alpha}\right)\right\}_{\alpha}$ has a convergent subsequence. Let $A$ be its limit. Then, 
\begin{gather*}
\liminf_{\alpha\rightarrow \infty} H\left( x_{\alpha}, \alpha \psi_x\left(x_{\alpha}, y_{\alpha}\right)\right) -  H\left( y_{\alpha}, \alpha \psi_x\left(x_{\alpha}, y_{\alpha}\right)\right)
\leq H\left(z, A\right) -  H\left(z, A\right) =0.
\end{gather*}
For $z=1$, we repeat the previous analysis, being careful with cases in which $x_{\alpha}=1$ or $y_{\alpha}=1$ after a certain $\alpha_0$.
\end{proof}

\subsection{Step 4: Variational representation of the rate function} \label{step4}

In this section, we formally define the Nisio semigroup $\mathcal{V}_t$ that was mentioned in Theorem \ref{thm:variational_version}. By Theorem \ref{thm:variational_version}, it is enough to prove that Conditions {\bf 8.9}, {\bf 8.10}, and {\bf 8.11} from \cite{F&K} are verified in our case. We present them as propositions. Als, the role of absolutely continuous functions in the definition of the rate function $I$ is explained.

\begin{definition}[Control set of a linear operator and Nisio semigroup] Let $U$ and $E$ be complete and separable metric spaces. Let $A: Dom(A)\subset B(E) \rightarrow M\left(E\times U\right)$ be a single valued linear operator. Let $\mathcal{M}_m (U)$ be the space of Borel measures $\lambda$ on $U\times [0,1]$ satisfying $\lambda\left(U\times [0,t]\right) = t$ for all $t\in [0,1]$. The measure $\lambda$ is known as a \emph{relaxed control}. We say that the pair $\left(\mathbf{x}, \lambda\right) \in D_{E}[0,1] \times \mathcal{M}_m(U)$ satisfies the \emph{ relaxed control equation } for $A$ if and only if:
\begin{enumerate}
\item $\iint_{U\times [0,t]} \left\vert A(f)(\mathbf{x}(s), u)\right\vert \lambda\left(\text{d}u\times \text{d}s\right) <\infty \; \forall f\in Dom(A), \; \forall t\in[0,1]$;\\
\item $f\left(\mathbf{x}(t)\right)-f\left(\mathbf{x}(0)\right)=\iint_{U\times [0,t]}  A(f)(\mathbf{x}(s), u) \lambda\left(\text{d}u\times \text{d}s\right) \; \forall f\in Dom(A), \; \forall t \in [0,1]$.
\end{enumerate} 
We denote the collection of pairs satisfying the above properties by $\mathcal{Y}$. The Nisio semigroup corresponding to the control problem determined by the linear operator $A$ and the cost function $-L$ is:
\begin{gather}\label{eq:Vt}
\mathcal{V}_t (f)(x_0)= \sup_{\{\left(\mathbf{x},\lambda \right) \in \mathcal{Y}: \; \mathbf{x}(0)=x_0\}} 
\left\{f(\mathbf{x}(t)) - \iint_{U\times [0,t]}  L(\mathbf{x}(s), u) \lambda\left(\text{d}u\times \text{d}s\right) \right\}
\end{gather}
for each $x_0 \in E$ (the supremum of an empty set is defined to be $-\infty$). Note that operator $A$ appears in the definition of the control set.
\end{definition}

In our case, as $\mathbf{H}(f)(x)= H \left(x, f'(x)\right)$ for each $x\in E= \left[0,1\right]$ and $H \leftrightarrow L$, we have that $\mathbf{H}$ can be written as 
$
\mathbf{H}(f)(x)= \underset{u \in U}{\sup} \left\{ A(f)(x,u)-L(x,u) \right\}, 
$
where $U= \mathbb{R}$ and $A:C^1(E)\rightarrow M\left(E\times U\right)$ is the linear operator $A(f)(x,u)=f'(x)u$. As $L$ is convex w.r.t. $\beta$, it follows that a deterministic control $\lambda\left(\text{d}u\times \text{d}s\right)=\delta_{u(s)}(\text{d}u)\text{d}s$ is allways the control with smallest cost by Jensen's inequality. Moreover, if $\mathbf{x}: E\rightarrow \mathbb{R}$ is an absolutely continuous function (we note $\mathbf{x}\in \mathcal{AC}$ for short), then
\begin{gather*}
f\left(\mathbf{x}(t)\right)-f\left(\mathbf{x}(0)\right)= \intop_{0}^{t}f'\left(\mathbf{x}(s)\right) \dot{\mathbf{x}}(s)\text{d}s= \iint_{\mathbb{R}\times [0,t]} f'\left(\mathbf{x}(s)\right) \, u \lambda(\text{d}u \times \text{d}s),
\end{gather*}
if we define  $\lambda= \lambda(\mathbf{x})$ such that $\lambda(\text{d}u \times \text{d}s)= \delta_{\dot{\mathbf{x}}(s)}(\text{d}u) \text{d}s$. Then, the supremum in Equation \ref{eq:Vt} is reached on $\left\{ (\mathbf{x}, \lambda): \, \mathbf{x} \in \mathcal{AC}, \; \mathbf{x}(0)=x_0 \right\} \subset \mathcal{Y}$.
\begin{proposition} Conditions {\bf 8.9} of \cite{F&K} are verified. 
\end{proposition}

\begin{proof}
Conditions (1) to (4) are trivially verified. For (5) we construct $\Psi_f$ as in Lemma {\bf 10.21} of \cite{F&K}.
\end{proof}
\begin{proposition}
Condition \textbf{8.10} from \cite{F&K} is verified, i.e. for all $x_0 \in E$ there exists $(\mathbf{x}, \lambda ) \in \mathcal{Y}$ such that $\mathbf{x}(0)=x_0$ and
\begin{gather*}
\iint_{U \times [0,1]} L\left( \mathbf{x}(s), u\right) \lambda \left(\text{d}u \times \text{d}s \right)= 0.
\end{gather*}
\end{proposition}
\begin{proof}
Since $L(x,u)=0 \Leftrightarrow u= H_{\alpha}(x,0)$,  the function  $q(x)=H_{\alpha}(x,0)= 1+c(1-x)$ solves the equation $L(x, q(x))=0$ for all $x\in E$.  Note that the fluid limit verifies $\dot{z}= q(z)$ with the initial condition $z(0)=0$. Given $x_0 \in E$, there exists $t_0 \in [0,1]$ such that $z(t_0)= x_0$, then $z\left(t+t_0 \right)$ is the solution of $\dot{x}= q(x)$ with $x(0)=x_0$. Define $\mathbf{x}(t)= z(t+t_0)\wedge 1$ for all $t\in[0,1]$ and $\lambda$ such that $\lambda\left(\text{d}u \times \text{d}s\right)= \delta_{\left\{q(\mathbf{x}(s))\right\}}(\text{d}u)\times \text{d}s$, then $\left(\mathbf{x}, \lambda \right)\in \mathcal{Y}$ and verifies the required condition. 
\end{proof}
\begin{proposition}
Condition {\bf 8.11} from \cite{F&K} is verified, i.e. for all $x_0 \in E$ and $f\in C^1(E)$ there exists $(\mathbf{x}, \lambda )=(\mathbf{x}_f, \lambda_f ) \in \mathcal{Y}$ such that $\mathbf{x}(0)=x_0$ and
\begin{gather*}
\intop_{t_1}^{t_2} \mathbf{H}(f)\left(\mathbf{x}(s)\right)\text{d}s \leq \iint_{U\times [t_1, t_2]} \left[ A(f)\left(\mathbf{x}(s), u \right) - L\left(\mathbf{x}(s),u \right) \right] \lambda\left(\text{d}s\times \text{d}u \right),
\end{gather*}
for all $0\leq t_1\leq t_2\leq 1$.
\end{proposition}
\begin{proof}
Let $x_0 \in E$ and $f\in C^1(E)$ be fixed. Since $\mathbf{H}(f)(x)\geq A(f)(x,u)-L(x,u)$ for all $(x,u)\in E\times U$, we need to find $(\mathbf{x}, \lambda ) \in \mathcal{Y}$ such that $\mathbf{x}(0)=x_0$ and 
\begin{equation}\label{A}
\intop_{0}^{t}
 H\left(\mathbf{x}(s), f'\left(\mathbf{x}(s)\right) \right)\text{d}s = \iint_{U\times [0,t]} \left[ f'\left(\mathbf{x}(s)\right)u - L\left(\mathbf{x}(s),u \right) \right] \lambda\left(\text{d}s\times \text{d}u \right),
\end{equation}
for all $t\in [0,1]$. If we define $q_f(x)=H_{\alpha}(x, f'(x))$, then $H(x, f'(x))=f'(x)q_f(x)-L\left(x, q_f(x)\right)$ and Equation \ref{A} is verified if we take $\lambda\left(\text{d}u \times \text{d}s \right)= \delta_{\left\{q_f\left(\mathbf{x}(s)\right)\right\}}(\text{d}u) \text{ d}s$. Now we have to add conditions so that in addition $\left(\mathbf{x}, \lambda\right)\in \mathcal{Y}$. In particular, $\left(\mathbf{x}, \lambda\right)$ has to verify:
\begin{gather*}
\intop_{0}^{t}g'\left(\mathbf{x}(s)\right) q_f\left(\mathbf{x}(s)\right)\text{d}s =  g\left(\mathbf{x}(t)\right)-g\left(\mathbf{x}(0)\right) \; \forall t\in [0,1], \; \forall g\in C^1(E).
\end{gather*}
Then we look for a path that solves the following problem:
\begin{equation} \label{eq:B}
\begin{cases}
\mathbf{x} \text{ is differentiable almost everywhere and } \dot{\mathbf{x}}(t)=q_f\left(\mathbf{x}(t)\right),\\
\mathbf{x}(0)=x_0,\\
\mathbf{x}(t) \in [0,1]  \text{ for all } t\geq 0.
\end{cases}
\end{equation}
Let $x_0 \in [0,1)$. Note that $q_f(x) = 1+c(1-x)e^{f'(x)}>1$ is continuous, then from Peano's theorem (see \cite{Crandall72}) we know that the ODE $\begin{cases} \dot{\mathbf{x}}(t)=q_f\left(\mathbf{x}(t)\right)\\ \mathbf{x}(t_0)=x_0 \end{cases}$ has a local solution $\mathbf{x}:J \rightarrow \R$, being $J$ an open neighborhood of $t_0$, it is also increasing and verifies $\mathbf{x}(t)\geq t$ for all $t\in [0,1]$. Since we need $\mathbf{x}\in D_E[0,1]$, we can paste these local solutions until the time $T_{x_0}$ it reaches $1$ and define $\mathbf{x}(t)=1$ for $ T_{x_0} \leq t \leq 1$. If $x_0=1,$ we take $\mathbf{x} \equiv 1$.
\end{proof}

\section*{Acknowledgements}
The authors would like to thank the referees for their valuable comments that help us to improve our manuscript significantly.
\bibliographystyle{alea3}

\bibliography{References}

\begin{thebibliography}{41}
\providecommand{\natexlab}[1]{#1}
\providecommand{\url}[1]{\texttt{#1}}
\providecommand{\urlprefix}{URL }
\expandafter\ifx\csname urlstyle\endcsname\relax
  \providecommand{\doi}[1]{doi:\discretionary{}{}{}#1}\else
  \providecommand{\doi}{doi:\discretionary{}{}{}\begingroup
  \urlstyle{rm}\Url}\fi
\providecommand{\eprint}[2][]{\url{#2}}

\bibitem[{de~Acosta(1997)}]{deAcosta}
A.~de~Acosta.
\newblock Exponential tightness and projective systems in large deviation
  theory.
\newblock In \emph{Festschrift for {L}ucien {L}e {C}am}, pages 143--156.
  Springer, New York (1997).
\newblock \doi{10.1007/978-1-4612-1880-7_9}.
\newblock \urlprefix\url{https://doi.org/10.1007/978-1-4612-1880-7_9}.

\bibitem[{Arnold(1987)}]{Arnold}
B.~I. Arnold.
\newblock \emph{M\'etodos {M}atem\'aticos da {M}ecanica {C}l\'asica.}
\newblock Editora {M}ir {M}oscovo (1987).

\bibitem[{Baccelli and Tien~Viet(2012)}]{Baccelli-12}
F.~Baccelli and N.~Tien~Viet.
\newblock Generating {F}unctionals of {R}andom {P}acking {P}oint {P}rocesses:
  {F}rom {H}ard-{C}ore to {C}arrier {S}ensing.
\newblock \emph{Computing {R}esearch {R}epository ({C}o{RR}).}  (2012).
\newblock \href{http://arxiv.org/abs/1202.0225}{abs/1202.0225}.

\bibitem[{Bermolen et~al.(2017{\natexlab{a}})Bermolen, Jonckheere and
  J.}]{Bermolen2017bounds}
P.~Bermolen, M.~Jonckheere and Sanders J.
\newblock Scaling {L}imits and {G}eneric {B}ounds for {E}xploration
  {P}rocesses.
\newblock \emph{Journal of {S}tatistical {P}hysics.} \textbf{5}, 989--1018
  (2017{\natexlab{a}}).

\bibitem[{Bermolen et~al.(2017{\natexlab{b}})Bermolen, Jonckheere and
  Moyal}]{Bermolen2017}
P.~Bermolen, M.~Jonckheere and Pascal. Moyal.
\newblock The jamming constant of uniform random graphs.
\newblock \emph{Stochastic {P}rocesses and their {A}pplications.} \textbf{7},
  2138--2178 (2017{\natexlab{b}}).

\bibitem[{Coste and Salez(2018)}]{salez}
Simon Coste and Justin Salez.
\newblock Emergence of extended states at zero in the spectrum of sparse random
  graphs (2018).
\newblock \eprint{1809.07587}.

\bibitem[{Cram\'er(1938)}]{Cramer}
H.~Cram\'er.
\newblock Sur un nouveau th\'eoreme-limite de la theorie des probabiliti\'es.
\newblock \emph{Acta. {S}ci. et {I}nd.} \textbf{736}, 5--23 (1938).

\bibitem[{Crandall and Liggett(1971)}]{Cra&Lig}
M.~G. Crandall and T.~M. Liggett.
\newblock Generation of semi-groups of nonlinear transformations on general
  {B}anach spaces.
\newblock \emph{Amer. J. Math.} \textbf{93}, 265--298 (1971).
\newblock ISSN 0002-9327.
\newblock \doi{10.2307/2373376}.
\newblock \urlprefix\url{https://doi.org/10.2307/2373376}.

\bibitem[{Crandall(1972)}]{Crandall72}
Michael~G. Crandall.
\newblock A generalization of {P}eano's existence theorem and flow invariance.
\newblock \emph{Proc. Amer. Math. Soc.} \textbf{36}, 151--155 (1972).
\newblock ISSN 0002-9939.
\newblock \doi{10.2307/2039051}.
\newblock \urlprefix\url{https://doi.org/10.2307/2039051}.

\bibitem[{Crandall et~al.(1992)Crandall, Ishii and Lions}]{CIL}
Michael~G. Crandall, Hitoshi Ishii and Pierre-Louis Lions.
\newblock User's guide to viscosity solutions of second order partial
  differential equations.
\newblock \emph{Bull. Amer. Math. Soc. (N.S.)} \textbf{27}~(1), 1--67 (1992).
\newblock ISSN 0273-0979.
\newblock \doi{10.1090/S0273-0979-1992-00266-5}.
\newblock \urlprefix\url{https://doi.org/10.1090/S0273-0979-1992-00266-5}.

\bibitem[{Dembo and Zeitouni(1998)}]{Dembo}
Amir Dembo and Ofer Zeitouni.
\newblock \emph{Large deviations techniques and applications}, volume~38 of
  \emph{Applications of Mathematics (New York)}.
\newblock Springer-Verlag, New York, second edition (1998).
\newblock ISBN 0-387-98406-2.
\newblock \doi{10.1007/978-1-4612-5320-4}.
\newblock \urlprefix\url{https://doi.org/10.1007/978-1-4612-5320-4}.

\bibitem[{Donsker and Varadhan(1975)}]{DonVar}
M.~D. Donsker and S.~R.~S. Varadhan.
\newblock Asymptotic evaluation of certain {M}arkov process expectations for
  large time. {I}, {II} (and {III}).
\newblock \emph{Comm. Pure Appl. Math.} \textbf{28}, 1--47; ibid. 28 (1975),
  279--301 (1975).
\newblock ISSN 0010-3640.
\newblock \doi{10.1002/cpa.3160280102}.
\newblock \urlprefix\url{https://doi.org/10.1002/cpa.3160280102}.

\bibitem[{Dupuis and Ellis(1997)}]{Dupuis}
Paul Dupuis and Richard~S. Ellis.
\newblock \emph{A weak convergence approach to the theory of large deviations}.
\newblock Wiley Series in Probability and Statistics: Probability and
  Statistics. John Wiley \& Sons, Inc., New York (1997).
\newblock ISBN 0-471-07672-4.
\newblock \doi{10.1002/9781118165904}.
\newblock \urlprefix\url{https://doi.org/10.1002/9781118165904}.
\newblock A Wiley-Interscience Publication.

\bibitem[{Dupuis et~al.(1991)Dupuis, Ellis and Weiss}]{Dupuis_discI}
Paul Dupuis, Richard~S. Ellis and Alan Weiss.
\newblock Large deviations for {M}arkov processes with discontinuous
  statistics. {I}. {G}eneral upper bounds.
\newblock \emph{Ann. Probab.} \textbf{19}~(3), 1280--1297 (1991).
\newblock ISSN 0091-1798.
\newblock
  \urlprefix\url{http://links.jstor.org/sici?sici=0091-1798(199107)19:3<1280:LDFMPW>2.0.CO;2-D&origin=MSN}.

\bibitem[{El~Karoui et~al.(1982)El~Karoui, Lepeltier and Marchal}]{ElKaroui}
N.~El~Karoui, J.P. Lepeltier and B.~Marchal.
\newblock Nisio semi-group associated to the control of {M}arkov processes.
\newblock \emph{Stochastic {D}ifferential {S}ystems.} \textbf{43} (1982).

\bibitem[{Evans(1993)}]{Evans1993}
J.W. Evans.
\newblock Random and cooperative sequential adsorption.
\newblock \emph{Reviews of {M}odern {P}hysics.} \textbf{65}~(4), 1281--1329
  (1993).
\newblock \href{https://link.aps.org/doi/10.1103/RevModPhys.65.1281}.

\bibitem[{Evans and Ishii(1985)}]{Evans_Ishii}
L.~C. Evans and H.~Ishii.
\newblock A {PDE} approach to some asymptotic problems concerning random
  differential equations with small noise intensities.
\newblock \emph{Ann. Inst. H. Poincar\'{e} Anal. Non Lin\'{e}aire}
  \textbf{2}~(1), 1--20 (1985).
\newblock ISSN 0294-1449.
\newblock \urlprefix\url{http://www.numdam.org/item?id=AIHPC_1985__2_1_1_0}.

\bibitem[{Feng and Kurtz(2006)}]{F&K}
Jin Feng and Thomas~G. Kurtz.
\newblock \emph{Large deviations for stochastic processes}, volume 131 of
  \emph{Mathematical Surveys and Monographs}.
\newblock American Mathematical Society, Providence, RI (2006).
\newblock ISBN 978-0-8218-4145-7; 0-8218-4145-9.
\newblock \doi{10.1090/surv/131}.
\newblock \urlprefix\url{https://doi.org/10.1090/surv/131}.

\bibitem[{Ferrari et~al.(2002)Ferrari, Fern\'{a}ndez and
  Garcia}]{ferrari-fernandez-garcia-02}
Pablo~A. Ferrari, Roberto Fern\'{a}ndez and Nancy~L. Garcia.
\newblock Perfect simulation for interacting point processes, loss networks and
  {I}sing models.
\newblock \emph{Stochastic Process. Appl.} \textbf{102}~(1), 63--88 (2002).
\newblock ISSN 0304-4149.
\newblock \doi{10.1016/S0304-4149(02)00180-1}.
\newblock \urlprefix\url{https://doi.org/10.1016/S0304-4149(02)00180-1}.

\bibitem[{Fleming(1977/78)}]{Fleming78}
Wendell~H. Fleming.
\newblock Exit probabilities and optimal stochastic control.
\newblock \emph{Appl. Math. Optim.} \textbf{4}~(4), 329--346 (1977/78).
\newblock ISSN 0095-4616.
\newblock \doi{10.1007/BF01442148}.
\newblock \urlprefix\url{https://doi.org/10.1007/BF01442148}.

\bibitem[{Fleming(1985)}]{Fleming}
Wendell~H. Fleming.
\newblock A stochastic control approach to some large deviations problems.
\newblock In \emph{Recent mathematical methods in dynamic programming ({R}ome,
  1984)}, volume 1119 of \emph{Lecture Notes in Math.}, pages 52--66. Springer,
  Berlin (1985).
\newblock \doi{10.1007/BFb0074780}.
\newblock \urlprefix\url{https://doi.org/10.1007/BFb0074780}.

\bibitem[{Fleming(1999)}]{Fleming99}
W.H. Fleming.
\newblock Stochastic analysis, control optimization and applications.  (1999).
\newblock A volume in honor of W.H. Fleming.

\bibitem[{Freidlin and Wentzell(1984)}]{F&W}
M.~I. Freidlin and A.~D. Wentzell.
\newblock \emph{Random perturbations of dynamical systems}, volume 260 of
  \emph{Grundlehren der Mathematischen Wissenschaften [Fundamental Principles
  of Mathematical Sciences]}.
\newblock Springer-Verlag, New York (1984).
\newblock ISBN 0-387-90858-7.
\newblock \doi{10.1007/978-1-4684-0176-9}.
\newblock \urlprefix\url{https://doi.org/10.1007/978-1-4684-0176-9}.
\newblock Translated from the Russian by Joseph Sz\"{u}cs.

\bibitem[{Gamarnik and Sudan(2017)}]{Gamarnik}
David Gamarnik and Madhu Sudan.
\newblock Limits of local algorithms over sparse random graphs.
\newblock \emph{Ann. Probab.} \textbf{45}~(4), 2353--2376 (2017).
\newblock ISSN 0091-1798.
\newblock \doi{10.1214/16-AOP1114}.
\newblock \urlprefix\url{https://doi.org/10.1214/16-AOP1114}.

\bibitem[{Hatami et~al.(2014)Hatami, Lov\'{a}sz and Szegedy}]{Hatami}
Hamed Hatami, L\'{a}szl\'{o} Lov\'{a}sz and Bal\'{a}zs Szegedy.
\newblock Limits of locally-globally convergent graph sequences.
\newblock \emph{Geom. Funct. Anal.} \textbf{24}~(1), 269--296 (2014).
\newblock ISSN 1016-443X.
\newblock \doi{10.1007/s00039-014-0258-7}.
\newblock \urlprefix\url{https://doi.org/10.1007/s00039-014-0258-7}.

\bibitem[{Jonckheere and Saenz(2019)}]{Jonckheere}
M.~Jonckheere and M.~Saenz.
\newblock Asymptotic optimality of degree-greedy discovering of independent
  sets in configuration model graphs  (2019).
\newblock \urlprefix\url{https://arxiv.org/abs/1808.10358}.

\bibitem[{Jungnickel(2005)}]{Jungnickel}
Dieter Jungnickel.
\newblock \emph{Chapter 5, The Greedy Algorithm, In:Graph, Networks and
  Algorithm}, volume~5 of \emph{Algorithm and Computation in Mathematics}.
\newblock Springer-Verlag Berlin Heidelberg, second edition (2005).
\newblock \urlprefix\url{https://doi.org/10.1007/b138283}.

\bibitem[{Karp and Sipser(1981)}]{Karp-Sipser}
R.~M. Karp and M.~Sipser.
\newblock Maximum matching in sparse random graphs.
\newblock In \emph{22nd Annual Symposium on Foundations of Computer Science
  (sfcs 1981)}, pages 364--375 (1981).
\newblock \doi{10.1109/SFCS.1981.21}.

\bibitem[{Kleinrock and Takagi(1985)}]{Kleinrock}
L.~Kleinrock and H.~Takagi.
\newblock Throughput analysis for persistent csma systems.
\newblock \emph{IEEE Transactions on Communications} \textbf{33}~(7), 627--638
  (1985).

\bibitem[{Kraaij(2016)}]{Kraaij}
Richard Kraaij.
\newblock Large deviations for finite state {M}arkov jump processes with
  mean-field interaction via the comparison principle for an associated
  {H}amilton-{J}acobi equation.
\newblock \emph{J. Stat. Phys.} \textbf{164}~(2), 321--345 (2016).
\newblock ISSN 0022-4715.
\newblock \doi{10.1007/s10955-016-1542-8}.
\newblock \urlprefix\url{https://doi.org/10.1007/s10955-016-1542-8}.

\bibitem[{McDiarmid(1990)}]{McDiarmid}
Colin McDiarmid.
\newblock Colourings of random graphs.
\newblock In \emph{Graph colourings ({M}ilton {K}eynes, 1988)}, volume 218 of
  \emph{Pitman Res. Notes Math. Ser.}, pages 79--86. Longman Sci. Tech., Harlow
  (1990).

\bibitem[{Nisio(1976)}]{Nisio76}
M.~Nisio.
\newblock On stochastic optimal controls and envelope of {M}arkovian
  semi-groups.
\newblock \emph{Proc. of. Intern. Symp. SDE} pages 297--325 (1976).

\bibitem[{Nisio(1978)}]{Nisio78}
M.~Nisio.
\newblock On non linear semi-groups for {M}arkov processes associated with
  optimal stopping.
\newblock \emph{Applied Math. and Opt.} \textbf{4}, 143--169 (1978).

\bibitem[{O'Brien and Vervaat(1995)}]{OV}
George~L. O'Brien and Wim Vervaat.
\newblock Compactness in the theory of large deviations.
\newblock \emph{Stochastic Process. Appl.} \textbf{57}~(1), 1--10 (1995).
\newblock ISSN 0304-4149.
\newblock \doi{10.1016/0304-4149(95)00007-T}.
\newblock \urlprefix\url{https://doi.org/10.1016/0304-4149(95)00007-T}.

\bibitem[{Penrose(2001)}]{Penrose}
Mathew~D. Penrose.
\newblock Random parking, sequential adsorption, and the jamming limit.
\newblock \emph{Comm. Math. Phys.} \textbf{218}~(1), 153--176 (2001).
\newblock ISSN 0010-3616.
\newblock \doi{10.1007/s002200100387}.
\newblock \urlprefix\url{https://doi.org/10.1007/s002200100387}.

\bibitem[{Pittel(1982)}]{Pittel}
B.~Pittel.
\newblock On the probable behaviour of some algorithms for finding the
  stability number of a graph.
\newblock \emph{Mathematical Proceedings of the Cambridge Philosophical
  Society} \textbf{92}~(3), 511–526 (1982).
\newblock \doi{10.1017/S0305004100060205}.

\bibitem[{Puhalskii(1994)}]{Puhalskii}
A.~Puhalskii.
\newblock The method of stochastic exponentials for large deviations.
\newblock \emph{Stochastic Process. Appl.} \textbf{54}~(1), 45--70 (1994).
\newblock ISSN 0304-4149.
\newblock \doi{10.1016/0304-4149(94)00004-2}.
\newblock \urlprefix\url{https://doi.org/10.1016/0304-4149(94)00004-2}.

\bibitem[{Ritchie(2006)}]{Ritchie}
T.~Ritchie.
\newblock Construction of the thermodynamic jamming limit for the parking
  process and other exclusion schemes on ${Z}^{d}$.
\newblock \emph{Journal of Statistical Physics} \textbf{122}, 381--398 (2006).

\bibitem[{Spitzer(1975)}]{Spitzer}
Frank Spitzer.
\newblock Markov random fields on an infinite tree.
\newblock \emph{Ann. Probability} \textbf{3}~(3), 387--398 (1975).
\newblock ISSN 0091-1798.
\newblock \doi{10.1214/aop/1176996347}.
\newblock \urlprefix\url{https://doi.org/10.1214/aop/1176996347}.

\bibitem[{Touchette(2009)}]{Touchette}
Hugo Touchette.
\newblock The large deviation approach to statistical mechanics.
\newblock \emph{Phys. Rep.} \textbf{478}~(1-3), 1--69 (2009).
\newblock ISSN 0370-1573.
\newblock \doi{10.1016/j.physrep.2009.05.002}.
\newblock \urlprefix\url{https://doi.org/10.1016/j.physrep.2009.05.002}.

\bibitem[{Wormald(1995)}]{wormaldDF}
Nicholas~C. Wormald.
\newblock Differential equations for random processes and random graphs.
\newblock \emph{Ann. Appl. Probab.} \textbf{5}~(4), 1217--1235 (1995).
\newblock ISSN 1050-5164.
\newblock
  \urlprefix\url{http://links.jstor.org/sici?sici=1050-5164(199511)5:4<1217:DEFRPA>2.0.CO;2-A&origin=MSN}.

\end{thebibliography}

\end{document}